\documentclass[a4paper, 11pt]{amsart}

\usepackage[utf8]{inputenx}
\usepackage[english]{babel}

\usepackage{enumerate}
\usepackage{amsthm}
\usepackage{amsmath}
\usepackage{amssymb}
\usepackage{scalerel}
\usepackage{mathtools}
\usepackage{bbm}
\usepackage{cases}
\usepackage{tikz}
\usepackage{IEEEtrantools}
\usepackage{hyperref}
\usepackage{graphicx}
\usepackage{emptypage}
\usepackage{color}
\usepackage{thmtools}
\usepackage{thm-restate}
\usepackage{todonotes}

\usepackage{comment}

\newtheorem*{thm*}{Theorem}
\newtheorem{thm}{Theorem}[section]

\newtheorem{lemma}[thm]{Lemma}
\newtheorem{prop}[thm]{Proposition}
\newtheorem{corollary}[thm]{Corollary}
\newtheorem{conjecture}[thm]{Conjecture}

\theoremstyle{definition}

\newtheorem{remark}[thm]{Remark}
\newtheorem{example}[thm]{Example}

\numberwithin{equation}{section}

\newenvironment{ieee*}[1]{\begin{IEEEeqnarray*}{#1}}{\end{IEEEeqnarray*}\ignorespacesafterend}
\newcommand{\f}{\varphi}
\newcommand{\la}{\lambda}

\newcommand{\med}{\medskip\noindent}

\newcommand{\ie}{\emph{i.e.}}

\newcommand{\R}{\mathbb{R}}
\newcommand{\Z}{\mathbb{Z}}
\newcommand{\N}{\mathbb{N}}

\newcommand{\K}{\mathcal{K}}

\newcommand{\One}{\mathbbm{1}}
\newcommand{\Prob}{\mathcal{P}}
\newcommand{\Meas}{\mathcal{M}}

\newcommand{\eps}{\varepsilon}
\newcommand{\wconv}{\rightharpoonup}
\newcommand{\Gconv}{\xrightarrow{\Gamma}}

\newcommand{\Rdn}{(\R^d)^N}

\newcommand{\eqdef}{\vcentcolon=}

\newcommand{\gra}[1]{\left \{ {#1} \right \}}
\newcommand{\st}{\;\colon}
\newcommand{\abs}[1]{\left| {#1} \right|}
\newcommand{\norm}[1]{\left\lVert {#1} \right\rVert}
\newcommand{\bra}[1]{\left\langle {#1} \right\rangle}
\newcommand{\pa}[1]{\left( {#1} \right)}
\newcommand{\floor}[1]{\lfloor {#1} \rfloor}
\newcommand{\ceil}[1]{\lceil {#1} \rceil}
\newcommand{\restr}{\lfloor}

\def\e{\varepsilon}
\def\O{\Omega}
\def\Ob{\overline{\Omega}}
\def\S{\mathcal{S}}
\def\L{\mathcal{L}}
\def\I{\mathcal{I}}
\def\ov{\overline}
\def\d{\delta}
\def\a{\alpha}

\newcommand{\res}{\mathbin{\vrule height 1.6ex depth 0pt width 0.13ex\vrule height 0.13ex depth 0pt width 1.3ex}}

\DeclareMathOperator{\supp}{supp}
\DeclareMathOperator{\cl}{cl}
\DeclareMathOperator{\co}{co}
\DeclareMathOperator{\Sym}{Sym}
\DeclareMathOperator*{\Glim}{\Gamma -lim}

\title{Relaxed many-body optimal transport and related asymptotics}

\begin{document}

\author{Ugo Bindini}
\address{Scuola Normale Superiore \\ Piazza dei Cavalieri, 7 \\ 56126 Pisa - ITALY}
\email{ugo.bindini@sns.it}

\date{\today}
\author{Guy Bouchitt\'e}
\address{ Imath \\ Universit\'e de Toulon, BP 20132\\  83957 La Garde Cedex- FRANCE}
\email{bouchitte@univ-tln.fr}

	\begin{abstract}  Optimization problems on probability measures in $\R^d$ are considered where
	the cost functional involves  multi-marginal optimal transport.  In a model of $N$  interacting particles, 
	like in Density Functional Theory, the interaction cost is repulsive and described by a two-point function 
	$c(x,y) =\ell(|x-y|)$ where $\ell: \R_+ \to [0,\infty]$ is decreasing to zero at infinity. Due to a possible loss of mass at infinity, non existence may occur and  relaxing the initial problem over  sub-probabilities becomes necessary.  
	In this paper we characterize the relaxed functional generalizing the results of \cite{bouchitte2020relaxed} and  
	present a duality method which allows to compute the $\Gamma-$limit as $N\to\infty$ under very general assumptions on the cost $\ell(r)$. We show that this  limit coincides with the convex hull of the so-called direct energy. Then we study the limit optimization problem when a continuous external potential is applied. Conditions are given with explicit examples under which minimizers are probabilities or have a mass $<1$ . 
	In a last part we study the case of a small range interaction $\ell_N(r)=\ell (r/\e)$ ($\e\ll 1$) and we show how the duality approach  can be  also used to determine the limit energy as $\e\to 0$ of a very large number $N_\e$ of particles. 	 	
	  
\end{abstract}

\maketitle



\maketitle
\section{Introduction}

Given a probability measure $\rho \in \Prob\left( \R^d \right)$ and $N \in \N$, $N \geq 2$, we consider the multi-marginal Optimal Transport (OT) problem defined by

\begin{equation} \label{eq:OTproblem}
  C_N(\rho) = \inf \gra{\int c_N(x_1, \dotsc, x_N) dP(x_1, \dotsc, x_N) \st P \in \Pi_N(\rho)}
\end{equation}
where $\Pi_N(\rho)$ denotes the set of multi-marginal transport plans 
\[ \Pi_N(\rho) \eqdef \gra{P \in \Prob\left( \Rdn \right) \st (\pi_j)_{\#}(P) = \rho \;\forall j = 1, \dotsc, N }. \]
being $\pi_j$ the projections from $\R^{Nd}$ on the $j$-th factor $\R^d$ and ${\pi_j}^\#$ the push-forward operator
\[ {\pi_j}^\#P(E)=P\big(\pi_j^{-1}(E)\big)\qquad\hbox{for all Borel sets }E\subset\R^d. \]

As a cost function $c_N$ we consider a two-particle interaction of the form
\begin{equation}\label{def:c_N}
  c_N(x_1, \dotsc, x_N) = \frac{2}{N(N-1)} \sum_{1\le i<j\le N} \ell(\abs{x_i-x_j})
\end{equation}
where $\ell \colon [0,+\infty] \to \R\cup \{+\infty\}$ satisfies the following standing assumptions:
\begin{enumerate}
  \item [(H1)]\  $\ell(r) \geq 0$ and $\ell(0)>0$;
  \item  [(H2)]\  $\ell$ is lower semi-continuous;
  \item [(H3)]\   $\displaystyle \lim_{r \to +\infty} \ell(r) = 0$;
\end{enumerate}
In addition to the standing assumptions, we will sometimes assume that
\begin{enumerate}
\item [(H4)] $\ell$ is locally integrable on $\R^d$, \ie, $\displaystyle \int_{\gra{|z|\le R}} \ell(\abs{z}) dz < +\infty, \forall R>0$; or
\item [(H5)] $\ell(0) < +\infty$ and $\ell$ is positive semi-definite in the following sense: for every $m\in \N^{*}$ and every subset $\gra{x_1, x_2, \dotsc, x_m} \subset (\R^d)^m$,
\[ \sum_{i,j=1}^m  \ell(|x_i-x_j|) t_i t_j \ge 0 \quad \forall t_1, t_2, \dotsc, t_m \in \R. \]
\end{enumerate}
Note that, by applying above with $m=2$, we get $\displaystyle\sup \ell = \ell(0) < +\infty$; hence (H5) implies (H4) while, as well known, the Fourier transform  of $\ell$ satisfies $\hat \ell\ge 0$ on $\R^d$. 

\medskip
It is common in many applications as in Density Functional Theory, crowd motion models and statistics, to encounter minimum problems of the form
\[ \inf_{\rho \in \Prob(\R^d)} \gra{C_N(\rho) + \mathcal{F}(\rho)}, \]
where $\mathcal{F}$ is a suitable density functional. In this context it is important to understand the behaviour of this value and the structure of the minimizers for a large number $N$ of particles/people, as this can be used to approximate the behaviour of large systems which are often out of reach  through numerical methods.

The first step in order to treat rigorously these instances is to understand the limit as $N \to \infty$ of the multi-marginal OT functional. In this setting, a natural tool is the notion of $\Gamma$-convergence with respect to the weak* topology of Radon measures on $\R^d$.  In particular if 
$\displaystyle C_\infty \eqdef \Glim_{n \to \infty} C_N$ exists and can be identified,
it will possible to pass to the limit in minimum problems of the kind  $\inf_{\rho} \gra{C_N(\rho) + \mathcal{F}(\rho)}$.  In particular, if $\mathcal{F}$ is  weakly continuous, by applying a celebrated theorem of De Giorgi, we will obtain the convergence of the infima
\[ \lim_{n \to \infty} \inf_{\rho} \gra{C_N(\rho) + \mathcal{F}(\rho)} = \min_{\rho} \gra{C_\infty(\rho) + \mathcal{F}(\rho)} \]
and the weak* convergence of minimizing sequences in $\Prob\left( \R^d \right)$ to  minimizers of the functional $ C_\infty + \mathcal{F}.$
However, as far as the minimum problem involves measures on the whole space $\R^d$,  this result  applies only in the case where  minimizing sequences are tight (\ie\   do not 
undergo a loss of mass at infinity).  Actually  such a condition rules out  very simple cases, as 
for instance when  $\mathcal{F}$ a linear functional of the form $\mathcal{F}(\rho)= \int v\, d\rho$ being $v$ a continuous potential with compact support.    
To overcome this difficulty,  we need to extend the $N$-particles problem and its limit $N\to\infty$ on the larger space $\Prob_-(\R^d)$ consisting of all sub-probabilities. This requires  a relaxion procedure
 for  the multi-marginal OT cost in the same line as in \cite{bouchitte2020relaxed}.

%

\medskip

The contributions of this  paper are  presented according to the following plan:

\medskip
In Section \ref{relaxation-section}, we extend to general costs $\ell$ the relaxation and duality framework recently developed in the case of the Coulomb interaction energy (see \cite{bouchitte2020relaxed}). As an application, assuming that $\ell(0) < +\infty$, we derive a recipe for computing the relaxed energy $\overline{C_N}(\rho)$ for $\rho$ being a finitely supported in $ \Prob_-(\R^d)$.  Explicit expressions are given 
when $\rho$ is a combination of two Dirac masses 

\medskip 
In Section \ref{sec:limits}, we  prove the $\Gamma$-convergence of  $C_N$ as $N\to\infty$ and provide a characterization of the limit functional $C_\infty: \Prob_-(\R^d)\to [0,+\infty]$. This representation
relies on a duality argument allowing to compute the Legendre-Fenchel conjugate of $C_\infty$ as a functional on $C_0(\R^d)$. This result stated in \autoref{MIformula} is the natural generalization of \cite{choquet1958diametre} for general two-point particle interaction.
We remark that the pointwise convergence of $C_N(\rho)$ for $\rho$ being a given probability was studied by B. Pass et al. in \cite{cotar2015infinite} \cite{pass2013optimal}  in  case of a positive definite cost function $\ell$. We refer also to the seminal work by G. Choquet in 1958  \cite{choquet1958diametre} and 
to recent works devoted to the next order asymptotics $C_N(\rho)$ \cite{petrache2017next,Serfaty2015, serfaty2018systems}. 
In all these works the simple limit on $\Prob(\R^d)$ is identified as the so called direct energy 
$$  D(\rho) = \int \ell(\abs{x-y}) d\rho(x) d\rho(y) ,$$
which, in the Coulomb's case ($d=3$ and $\ell(t)= t^{-1}$), represents the potential energy due to the self-interaction of the density of charge  $\rho$ with itself.

\medskip
In Section \ref{sec:direct-energy}, we tackle the natural question of the relation between the $\Gamma-$limit $C_\infty$,  the two-homogeneous extension $D_2$ of $D$ to $\Prob_-(\R^d)$  (see \eqref{Dalpha}) and the lower semicontinuous envelope $\ov{D}$ of $D$ (see \eqref{relaxD}). Under the general  assumptions on $\ell$ given above, we prove that $C_\infty$ agrees with the convex lower semicontinuous envelope  of $D_2$. In particular, if $D_2$ is convex,  we recover the equalities  $C_\infty= D_2= \ov{D} $ which are consistent with the common case where  $\ell$ is of positive type.

\medskip
In Section \ref{minimizers}, we specialize in the minimum problem
$$\inf_{\rho \in \Prob_-(\R^d)} \gra{C_\infty(\rho) -\lambda \bra{v,\rho}}\ ,$$
 where  $v\in C_0(\R^d)$ is a given external potential and $\lambda$  a positive parameter.
 Denoting by $\S_\la(v)$  the non empty set of solutions,
  we show the existence of threshold values $\la_*(v)\le \la^*(v)<+\infty$ such that $\S_\la(v)\subset \Prob(\R^d)$
 for $\la\ge \la^*(v)$, while $\S_\la(v)\cap\Prob(\R^d)$ is empty for $0\le \la <\la_*(v)$. An estimate of these tresholds in terms of the behavior of $v$ at infinity is provided in Subsection \ref{estimates} allowing to state that $\la_*(v) >0$ in many situations including the case of compactly supported $v$. On the other hand, for $\ell$ of positive type, we have  $C_\infty= D_2$ while the solution is unique \ie\ $\S_\la(v)=\{\rho_\la\}$ and $\la_*(v)=\la^*(v)$. In this case, the behavior
 of the map $\la \mapsto \rho_\la$ is shown to be linear below the common threshold. 
 These results are illustrated in the last Subsection \ref{radial} where  explicit solutions are given in dimension $d=3$ for a radial potential and $\ell(r)= \frac1{r}$ being the Coulomb cost. 
 
  

\medskip
In  Section \ref{crowd}, we investigate the possibility of extending our duality method to the case of a  varying cost $\ell_N$ which depends on a infinitesimal interaction distance parameter $\e_N$. This is motivated by the justification of the  passage from  discrete to continuous  in some multi-particle models  (mean-field limit),   as those that are used  for instance in crowd motion theory (see \cite{maury2018congested}).  In case of a hard spheres model, we are able to establish a complete asymptotic result which to our knowledge is new.

\medskip
Eventually  we postpone  to the Appendix some background on duality, convex analysis and $\Gamma$-convergence theory.
 
\subsection*{Notations}  Throughout the paper we will use the following notations: 
\setlength{\leftmargini}{6pt}
\begin{enumerate}

\item [-] $B(x,r)$ is the open ball of the Euclidean space $\R^d$ centered at $x$ and of radius $r$ ;

\item [-]  $C_0(\R^d)$ denotes the Banach space of continuous functions on $\R^d$ vanishing at infinity, $C_0^+(\R^d)$  the subspace of non negative ones. 
 
\item [-] $\Meas(\R^d)$ (resp.$\Meas_+(\R^d)$) stands for the space of signed Borel (resp. nonnegative measures) on $\R^d$;   $\Prob_-(\R^d)$ (resp. $\Prob(\R^d)$) is the subset of Borel measures $\mu\in \Meas_+(\R^d)$ such that $\|\mu\| := \mu(\R^d) \leq 1$ (resp. $\|\mu\| = 1$).

\item[-] Given $\mu \in \Meas(\R^d)$ and $h \in \R^d$,  $\tau_h \mu$ be the translation of $\mu$ by the vector $h$, \ie, $\tau_h \mu(E) = \mu(E-h)$ for every Borel set $E$.

\item [-] The bracket $\bra{\cdot, \cdot}$ will denote the duality between $C_0(\R^d)$ and $\Meas(\R^d)$:
\[ \bra{v, \mu} = \int v d\mu , \]
This duality pairing naturally induces the weak* topology on $\Meas(\R^d)$.

\item [-]  The topological support of $\mu \in \Meas_+(\R^d)$ is denoted  $\supp(\mu)$ 
while $\mu\res A$; represents its  trace on a Borel subset $A\subset \R^d$;   

\item [-]  If $v \in C_0(\R^d)$, we define $S_N v \colon \Rdn \to \R$ as
\[ S_N v(x_1, \dotsc, x_N) = \frac{1}{N} \sum_{j = 1}^N v(x_j) \]

\item [-] For a measure $\mu \in \Meas(\Rdn)$, we denote by $\Sym(\mu)$ its symmetrization, given by
\[ \Sym(\mu)(E) = \frac{1}{N!} \sum_{\sigma \in \mathfrak{S}_N} \mu(\sigma(E)), \]
where $\sigma(E) = \gra{(x_1, \dotsc, x_N) \st (x_{\sigma(1)}, \dotsc, x_{\sigma(N)}) \in E}$ for a permutation $\sigma \in \mathfrak{S}_N$.

\item[-]  $\Prob(\Prob_-(\R^d))$ denotes the set of Borel probabilities measures on $\Prob_-(\R^d)$ (seen as 
a weakly* compact metrizable space).

\end{enumerate}

\subsection*{Acknowledgements} The first author is grateful to the financial support of {\nobreakdash INdAM} (Istituto Nazionale d'Alta Matematica), via the LIA LYSM project.




\section{Relaxed multimarginal energy} \label{relaxation-section}

Let $\ell$ be a cost satisfying the standing assumptions. Then it is straightforward to check that the functional  $C_N \colon \Prob(\R^d) \to [0,+\infty]$ is convex, proper and lower semi-continuous on $\Prob(\R^d)$ endowed with the weak* topology. However, the latter property is not very useful in general when dealing with a sequence $(\rho_n)$ such that $\sup_n C_N(\rho_n) < +\infty$. Indeed, such a sequence may show-up a loss of mass at infinity and then be weakly converging to a sub-probability.

Following the direct method of calculus of variations, it is then natural to extend the definition of the multi-marginal energy to elements $\rho\in \Prob_-(\R^d)$ by introducing the relaxed cost:
\[ \ov{C_N}(\rho) = \inf \gra{ \liminf_n C_N(\rho_n) \st \rho_n \wconv \rho ,\ \rho_n \in \Prob(\R^d) } \]
Several characterizations of $\ov{C_N}(\rho)$ are available. Two of them, relying on a direct approach, are given in the next subsection
and will be used for some explicit computations (\autoref{Dirac}). A third one, very useful for studying the limit behavior of $C_N$ as $N\to\infty$, relies on duality theory and will be presented in Section \ref{sec:limits}.
%
%
%
\subsection{Two characterizations of \texorpdfstring{$\ov{C_N}$}{CN}} \label{2formulae}

A direct approach for characterizing $\overline{C_N}$ consists 
in embedding the elements $\rho\in\Prob_-$ as probabilities 
$\tilde{\rho} = i^\sharp(\rho) + (1 - \abs{\rho}) \delta_\omega$ over the Alexandrov's compactification 
$X = \R^d \cup \gra{\omega}$ of $\R^d$, where $\omega$ is the point at infinity 
and $i \colon x \mapsto x$ the identity embedding of $\R^d$ into $X$. The cost $c_N$ is then extended to $X^N$
by setting $\ell(\abs{a-b}) = 0$ whenever $a$ or $b$ equals $\omega$. Note that, with  this convention, 
the extension (still denoted $c_N$) is lower semi-continuous on $X$ thanks to (H3).
With these notations, a first expression for the relaxed cost $\ov{C_N}(\rho)$ in terms of $\tilde{\rho} = i^\sharp(\rho) + (1 -\|\rho\|) \delta_\omega$ can be derived similarly as in \cite[Proposition 2.2]{bouchitte2020relaxed}:

\begin{equation} \label{defCtild}
\ov{C_N}(\rho) \eqdef \min \gra{ \int_{X^N} c_N dP \st P \in \Prob(X^N), P \in \Pi(\tilde{\rho})}.
\end{equation}

Note that the existence of a minimum in \eqref{defCtild} follows from the lower semi-continuity of the map $P \mapsto \int_{X^N} c_N dP$ and of the compactness of $\Pi(\tilde{\rho})$ with respect to the narrow convergence on $\Prob(X^N)$.

Next considering \eqref{defCtild} and splitting the contributions of $\int_{X^N} c_N dP$ on each set of the form $(\R^d)^k \times {\omega}^{N-k}$  (see the proof of \cite[Theorem 2.3]{bouchitte2020relaxed}), we are led to  a second characterization of $\ov{C_N}$ namely:
\begin{equation} \label{relaxed-CN}
\overline{C_N}(\rho) = \inf_{\substack{a_1, \dotsc, a_N \geq 0 \\ \rho_1, \dotsc, \rho_N\in \Prob(\R^d)}} \gra{\sum_{i = 2}^N \frac{i(i-1)}{N(N-1)} a_i C_i(\rho_i) \st \sum_{i = 1}^N a_i \leq 1, \sum_{i = 1}^N \frac{i}{N} a_i \rho_i = \rho}.
\end{equation}

Let us mention that the expression above, which we call {\em stratification formula}, has, in a different context, some relationship with the grand-canonical formulation of optimal transport as it appears in \cite{cotar2019next}.
  
It is easy to check that the infimum in \eqref{relaxed-CN} is attained. Moreover, if $\|\rho\| \leq \frac{1}{N}$, we 
get $\overline{C_N}(\rho) = 0$ by choosing $a_1 = 1$, $a_2 = \dotsb = a_N = 0$ while, if $\rho \in \Prob(\R^d)$, the only possible choice $a_1 = \dotsb = a_{N-1} = 0$, $a_N = 1$ yields $\overline{C_N}(\rho) = C_N(\rho)$. More generally, the case of a fractional mass $\|\rho\| = \frac{K}{N}$ is interesting in the study of molecular structures, as it encodes a ionization phenomenon where exactly $K$ electrons among $N$ stay at finite distance, while the others $N-K$ go away to infinity. In this case, by taking all $a_i$ vanishing for $i \neq K$ and $a_K = 1$, we obtain the upper bound
\begin{equation}\label{trueineq}
 \overline{C_N}(\rho) \le \frac{K(K-1)}{N(N-1)} C_K \pa{\frac{N}{K} \rho}, \quad  \text{whenever $\|\rho\| = \frac{K}{N}$}.
\end{equation}

In view of the explicit formula given in \autoref{two-dirac}, it turns out that the inequality above is in fact an equality when $\rho$ is a combination of two Dirac masses and $\ell(0) < +\infty$. We expect that the equality holds true also in a larger class of $\rho$  under suitable conditions on the cost function $\ell$. For further comments and a conjecture related to this issue, we refer to \autoref{conj}.

\subsection{The vanishing gap conjecture} \label{conj}

Let us fix $\rho \in \Prob(\R^d)$ and $\theta \in [0,1]$, and apply \eqref{relaxed-CN} to the evaluation of $\overline{C_N}(\theta\rho).$
In an optimal choice $a_1, \dotsc, a_N$ there will be a minimum and a maximum index $i$ such that $a_i > 0$. Among all possible optimal choices, we select the one for which the difference between maximum and minimum index is the lowest, and we denote those indices respectively by $\overline{K}(N,\theta\rho)$ and $\underline{K}(N,\theta\rho)$. Observe that
\[ \sum_{i = 1}^N \frac{i}{N} a_i \rho_i = \theta\rho \implies \sum_{i = 1}^N i a_i = \theta N, \]
hence necessarily $\underline{K}(N,\theta\rho) \leq \theta N \leq \overline{K}(N,\theta\rho)$.

\begin{conjecture} \label{k-conjecture} The gap between $\underline{K}(N,\theta\rho)$ and $\overline{K}(N,\theta\rho)$ vanishes (with respect to $N$) as $N \to \infty$, \ie,
	\begin{equation}\label{gapvanish}
\limsup_{N \to \infty} \frac{\underline{K}(N, \theta\rho)}{N} = \liminf_{N \to \infty} \frac{\overline{K}(N, \theta\rho)}{N} = \theta. \end{equation}

\end{conjecture}

This is a weak version of the statement
\[ \underline{K}(N, \theta\rho) = \floor{\theta N}, \quad \overline{K}(N, \theta\rho) = \ceil{\theta N}, \]
where $\floor{x}$ (resp. $\ceil{x}$) denotes the largest integer smaller (resp. the smallest integer greater) than $x$. 
Having a gap bounded not larger than $1$ is true for a two points supported $\rho$ according to Subsection \ref{Dirac}, but 
possibly not true in general. Estimates on this gap are available in the recent work  by S. Di Marino, M. Lewin and L. Nenna
\cite{dimarino2022grandcanonical}. An immediate consequence of the validity of the asymptotic statement \eqref{gapvanish} will appear in the forthcoming \autoref{homogeneity-thm}.

\subsection{A recipe for computing \texorpdfstring{$\overline{C_N}$}{CN}}\label{Dirac}

An explicit computation of the relaxed transport cost $\overline{C_N}(\rho)$ for a general sub-probability $\rho \in \Prob_-(\R^d)$ is often much involved, and in most cases impossible to carry out. In this subsection we propose a recipe for deriving $\overline{C_N}(\rho)$ when $\rho$ is a combination of Dirac masses assuming that $\ell(0) < +\infty$.

\medskip
Let $\Sigma \eqdef \gra{x_1, x_2, \dotsc, x_m} \subset (\R^d)^m$ be a finite set supporting the atomic measure $\rho$. We associate with such $\Sigma$ the quadratic form defined by
$$ q_\Sigma(t) \eqdef \sum_{i, j=1}^m  \ell(|x_i-x_j|) t_i t_j, \quad t = (t_1, \dotsc, t_m) \in \R^m. $$
Then we define the convex function
\begin{equation}\label{fSiN}
f_\Sigma^{(N)} (t) \eqdef \inf_{\gamma \in \Prob(I_m^{(N)})} \gra{\sum_{k \in I_m^{(N)}} \gamma(k) q_\Sigma(k) \st \sum_{k \in I_m^{(N)}} \gamma(k) k = t}, 
\end{equation}
where $I_m^{(N)} \eqdef \gra{k \in \N^m: |k| \eqdef \sum_{i=1}^m k_i \le N}$ and we implicitly assume that $f_\Sigma^{(N)} (t) = +\infty$ if $t \notin \Delta_m^{(N)}$, being $\Delta_m^{(N)} \eqdef \R_+^d \cap \{|t|\le N\}$.

It is easy to check that $f_\Sigma^{(N)}$ coincides with the largest convex l.s.c. function $g \colon \R^d \to [0,+\infty]$ such that $g = q_{\Sigma}$ on $I_m^{(N)}$. Moreover, $f_\Sigma^{(N)}(t)$ is non-increasing with respect to $N$ with a limit as $N\to\infty$  given by 
\begin{equation} \label{fSinfty}
f_\Sigma(t) = \inf_{\gamma \in \Prob(\N^m)} \gra{ \sum_{k\in \N^m} \gamma(k) q_\Sigma(k) \st \sum_{k\in \N^m} \gamma(k) k = t}.
\end{equation}

On the other hand, due to the finiteness of the set $I_m^{(N)}$, the infimum in the linear programming  problem \eqref{fSiN} is actually a minimum and the convex funcion $f_\Sigma^{(N)} (t)$ is expected to be piecewise affine in $\Delta_N$. 

\begin{prop} \label{explicit}
	Let $\ell$ satisfy the standing assumptions and $\ell(0) < +\infty$. Let $\Sigma \eqdef \gra{x_1, x_2,\dots, x_m}$ and $s \in [0,1]^m$ such that $|s|:= \sum_i s_i \le 1$. Then, for every $N \geq 2$ :
	\begin{equation} \label{explidis}
	\overline{C_N} \pa{\sum_{i=1}^m s_i \delta_{x_i}} = \frac{f_\Sigma^{(N)} (Ns)}{N(N-1)} - \frac{\ell(0)}{N-1} \sum s_i. 
	\end{equation}
	In particular, if $q_\Sigma$ is non negative, then for every $k \in I_m^{(N)}$: 
	\begin{equation}\label{formula-N}
	\overline{C_N} \pa{\sum_{i=1}^m \frac{k_i}{N} \delta_{x_i}} = \frac{\sum_{i=1}^m k_i(k_i-1)}{N(N-1)} \ell(0) + \frac{2}{N(N-1)} \sum_{i \neq j} k_i k_j \ell(|x_i-x_j|).
	\end{equation}
\end{prop}

\begin{proof}
By the characterization in \eqref{defCtild}, $\overline{C_N}(\rho)$ is the minimal cost of symmetric $N$-transport plan $P_N$ on the compactified space $ X = \R^d \cup \gra{\omega}$ with marginals $\tilde{\rho}=\sum_{i=1}^m s_i \delta_{x_i} + (1-|s|) \delta_\omega$. 
The generic form of such a $N$-transport plan $P$ is given by
\[ P = \sum_{k \in I_m^{(N)}} \gamma(k) p_\Sigma(k) \quad \text{being} \quad p_\Sigma(k) \eqdef \Sym \left( \delta_{x_1}^{\otimes k_1} \otimes \dotsb\otimes \delta_{x_m}^{\otimes k_m} \otimes \delta_\omega^{\otimes (N-|k|)}\right)
\]
and  where $\gamma (k) \geq 0$ satisfies $\sum_{k \in I_m^{(N)}} \gamma(k) = 1$ (\ie, $\gamma \in \Prob(I_m^{(N)})$) and the marginal condition
\begin{equation}\label{margi-k}
	\sum_{k \in I_m^{(N)}}  \gamma(k) k = Ns. 
\end{equation}

The  cost of such an admissible  $P$ reads $ \int_{X^d} c_N dP = \sum_{k\in I_m^{(N)}} \gamma(k) c_N(p_\Sigma(k)).$ After a careful computation one checks that the cost $c_N(p_\Sigma(k))$ is given up to the multiplicative factor $\frac{1}{N(N-1)}$ by
\[ \sum_{k \in I_m^{(N)}} \pa{\sum_{i=1}^m k_i(k_i - 1) \ell(0) + \sum_{i \neq j} k_ik_j\ell(|x_i-x_j|)} = q_\Sigma(k) - |k| \ell(0). \]

Taking into account \eqref{margi-k}, the total cost reduces to:
\[ \int_{X^d} c_N dP = \sum_{k\in I_m^{(N)}} \gamma(k) \frac{q_\Sigma(k)}{N(N-1)} - \frac{s}{N-1} \ell(0). \]

Searching for the minimal cost brings us to solving \eqref{fSiN} and to the expression \eqref{explidis} for $\overline{C_N} \pa{\sum_{i=1}^m s_i \delta_{x_i}}$. Now if $k$ denotes a particular element of $I_m^{(N)}$, by plugging the value $s = \frac{k}{N}$, we obtain
\[ \overline{C_N} \pa{\sum_{i=1}^m \frac{k_i}{N} \delta_{x_i}} = \frac{f_\Sigma^{(N)} (k) - \ell(0) |k|}{N(N-1)}. \]

In view of the definition of $f_\Sigma^{(N)}$ in \eqref{fSiN}, if we assume that $q_\Sigma\ge 0$, then by convexity it holds  $f_\Sigma^{(N)}(k) \ge q_\Sigma(k)$ while the opposite inequality is obvious since $k\in I_m^{N}$. The equality \eqref{formula-N} follows.
\end{proof} 

In view of \eqref{explidis}, it is possible to recover an optimal decomposition in the stratification formula \eqref{relaxed-CN} 
from an optimal $\gamma$ in \eqref{fSiN}, as in the following.
 
\begin{corollary} \label{|k|=K} 
	Let $\rho= \sum_{i=1}^m s_i\, \delta_{x_i}$ and let  $\gamma\in \Prob(I_m^{(N)})$ be optimal \eqref{fSiN}.
	Then, setting for every $K\in \{0,1,\dots ,N\}$:
	\begin{equation}\label{gammaK}
	 a_K \eqdef \sum_{\abs{k}=K} \gamma(k), \quad \rho_K \eqdef \sum_{\abs{k}=K} \frac{\gamma(k)}{a_K}\pa{\sum_{i=1}^m \frac{k_i}{K} \delta_{x_i}},
	\end{equation}
	we obtain an optimal decomposition $\rho = \sum_{K = 1}^N \frac{K}{N} a_K \rho_K$ in \eqref{explidis}. 
	
	Accordingly, with the notation of Subsection \ref{conj}), we have:
	\begin{equation} \label{Kpm}
	\underline{K}(\rho) = \min \{ \abs{k} \st k\in \supp(\gamma) \}, \quad \ov{K}(\rho) = \max \{ \abs{k} \st k\in \supp(\gamma)\}.
	\end{equation}
\end{corollary}

\begin{proof} The admissibility of $\gamma$ in \eqref{fSiN} implies that  $N s_i = \sum_{k\in I_m^{(N)}} k_i \gamma(k)$ for every $1 \le i\le m$. By splitting the sum over the subsets $\{k\in I_m^{(N)} \st \abs{k}=K\}$ for $0\le K\le N$, we derive that
 $\rho = \sum_{K = 1}^N \frac{K}{N} a_K \rho_K$ while $\sum_{K=0}^N a_K = 1$, obtaining an admissible decomposition for
 \eqref{explidis}. It follows that
 \[ \overline{C_N}(\rho) \le \sum_{K = 2}^N a_K \frac{K(K-1)}{N(N-1)} C_K(\rho_K). \]
 
The converse inequality is a consequence of the optimality of $\gamma$ in \eqref{fSiN}. We have:
\begin{align*}
\overline{C_N}(\rho) &= \sum_{K=1}^N  \pa{ \sum_{|k|=K} \frac{ q_\Sigma(k) -\ell(0) |k|}{N(N-1)} \gamma(k)} \\
&\ge  \sum_{K=2}^N  a_K \frac{K(K-1)}{N(N-1)}   \pa{\sum_{|k|=K} \frac{ q_\Sigma(k) -\ell(0) K}{K(K-1)} \gamma_K(k)} \\
&\ge \sum_{K=2}^N  a_K \frac{K(K-1)}{N(N-1)} C_K(\rho_K),
\end{align*} 
where in the second inequality  $\gamma_K$ denotes the probability on $\{|k|=K\}$  such that $\gamma_K(k)= \frac{\gamma(k)}{a_K}$. In the second inequality, we applied \eqref{explidis} with $N=K$ noticing that $\gamma_K$ is an admissible competitor in 
$\Prob(I_m^{(K)})$ for $f_{\Sigma}^{(K)}(t^{(K)})$ being $t^{(K)} = \sum_{|k|=K} \gamma_K(k) k$.
In turn, this implies that
\[ C_K(\rho_K) = C_K \left(\sum_{i=1}^m  \frac{t^{(K)}_i}{K} \delta_{x_i}\right) \leq \sum_{|k|=K} \frac{ q_\Sigma(k) -\ell(0) K}{K(K-1)} \gamma_K(k). \]
Summarizing, we proved that the decomposition of $\rho$ given by $a_K, \rho_K$ (as defined in \eqref{gammaK}) is optimal in the stratification formula \eqref{relaxed-CN}. The relations \eqref{Kpm} follows directly  from the expression of $a_K$.
\end{proof} 

\begin{remark}\label{N-quantized} 
 If $ q_\Sigma\ge 0$, the explicit formula \eqref{formula-N} for $\overline{C_N}(\rho)$ allows us to see  that the equality holds in \eqref{trueineq} 
namely $$\overline{C_N}(\rho)=\frac{K(K-1)}{N(N-1)}\, C_K \left(\frac{N}{K} \rho\right)$$ whenever
$\rho$ is a $N$-quantized sub-probability of total fractional mass $\frac{K}{N}$ , i.e. of the form $\sum_{i=1}^m \frac{k_i}{N}\, \delta_{x_i}$ for suitable $k_i$ such that $K=\sum k_i<N$.
The extension of previous equality to general $\rho$ such that $\|\rho\|= K/N$ is straightforward if the convex function $f_\Sigma^{(N)}$ agrees 
on $I_m^{(N)}$ with a suitable function on $\R_+^m$ not depending on $N$. Equivalently we need that $f_\Sigma^{(N)}= f_\Sigma$ on $I_m^{(N)}$
for every $N\ge 2$ with $f_\Sigma$ given by \eqref{fSinfty}.
 We expect this equality  to be true for any cost $\ell$ such that $q_\Sigma\ge 0$ although we could prove it only in the case $m=2$ by means of  a characterization of $f_\Sigma^{(N)}$ given below (see \autoref{m=2}) .
 Note that the sign condition on the quadratic form $q_\Sigma$ is satisfied for every finite set $\Sigma$ if and only if $\ell$ is a positive semi-definite function in the sense of (H5).
   \end{remark}
   
 \subsection{Identification of \texorpdfstring{ $f_\Sigma^{(N)}$}{SigmaN} (resp. \texorpdfstring{$ f_{\Sigma}$}{fSigma}) through simplicial partitions}\ 
 
In the spirit of the finite element method, we consider  partitions of  $\Ob=\Delta_m^{(N)}$ (resp. of $\Ob=\R_+^m$)
which consist of a family $\mathcal{F}$ of simplices $S$ in $\R^m$ such that:
\begin{itemize}
\item [i)]\ $\cup\{ S  : S\in \mathcal{F}\} = \Ob$; 
\item[ii)]\  the elements of $\mathcal{F}$ have  mutually disjoints interiors;
\item[iii)]\ any face of a simplex $S\in \mathcal{F}$ is either a face of another $S'\in \mathcal{F}$,  or a subset of $\partial\Ob$.
\end{itemize}
For the existence  of such partitions in any dimension $m$, we refer to \cite{Brandts2020} (see references therein).
In our case we will restrict ourselves to partitions $\mathcal{F}$,  that we call {\em admissible simplicial partitions}, consisting of simplices whose vertices are in the integer lattice $\N^m$.
 A natural guess supported by forthcoming \autoref{partition}
is that $f_\Sigma^{(N)}$ and $f_\Sigma$ should be piecewise affine along such a partition.

\begin{lemma} \label{partition}
	Assume that $q_\Sigma\ge 0$. Let be given an admissible simplicial partition of $\Delta_m^{(N)}$ (resp $\R_+^m$) and denote by $g$  the unique continuous function $g$ which is is affine on each simplice and satisfies  $g= q_\Sigma $ on $I_m^{(N)}$ (resp. on $\N^m$).
	
	Then $f_\Sigma^{(N)} = g$ on $\Delta_m^{(N)}$ (resp $f_\Sigma = g$ on $\R_+^m$) if and only if $g$ is convex on $\Delta_m^{(N)}$ (resp $\R_+^m$). If it is the case, then the interior of each simplex $S$ does not meet $\N^m$.
\end{lemma}

\begin{proof} By restricting the infimum in \eqref{fSiN} to elements $\gamma\in \Prob(I_m^{(N)})$ supported on the vertices of a simplex, we see that the inequality $f_\Sigma^{(N)} \le g$ is always true. Since  $g = q_\Sigma = f_\Sigma^{(N)} $ on $I_m^{(N)}$, the opposite inequality $g \geq f_\Sigma^{(N)}$ holding on $\Delta_m^{(N)}$ is equivalent to the convexity of $g$. The characterization of $f_\Sigma$ is obtained in the same way.

Assume now that the interior of a simplex $S$ of the partition contains an element of $\N^m$. Then the convex quadratic function $q_\Sigma$ would agree with the affine function $g$ not only on the vertices of $S$ but also at that  interior point forcing $q_\Sigma$ to be affine over the whole simplex. This is impossible unless $q_\Sigma$ vanishes identically. 
\end{proof}

\begin{remark} \label{flat}
	A consequence of \autoref{partition} is that, given an  admissible simplicial partition $\mathcal{F}$ of $\R_+^m$ such that the affine interpolant of $q_\Sigma$ is convex, then the equality $f_\Sigma^{(N)} = f_{\Sigma}$ is true provided that any simplex $S\in \mathcal{F}$ satisfies the flatness criterium  $\abs{|s|-|t|} \le 1$ for all $s,t \in S$. Indeed, under the this condition, any  such a simplex $S$ touching the interior of  $\Delta_m^{(N)}$ satisfies $S \subset \Delta_m^{(N)}$. In the same way, if $ \mathcal{F}$ is an admissible simplicial partition of $\Delta_m^{(N)}$ providing  a convex interpolant of $q_\Sigma$, then in view of \eqref{Kpm} we have the bound:
	\[ \overline{K}(N, \rho) - \underline{K}(N, \rho) \le \sup_{S \in \mathcal{F}} \max \gra{|s|-|t|: s,t \in S}, \]
	holding for every  $\rho\in \Prob_-(\R^d)$  supported on $\Sigma$.
\end{remark}

It turns out that, in the case of two Dirac masses (\ie, $m=2$ and $\Sigma=\gra{x,y}$), the equality $f_\Sigma^{(N)} = f_\Sigma$ holds in $I_2^{(N)}$ and we can identify the optimal triangle partition of $\R_+^2$ associated with $f_\Sigma$ . More precisely, let us denote, for every $k,l \in \N$, the square $ Y_{kl}:=[k, k+1]\times [l,l+1]$ that we split into the two triangles $ T_{kl}^- = \{(u,v)\in Y_{kl}: u + v \le k+l+1\}$ and $ T_{kl}^+ = \{(u,v)\in Y_{kl}: u+v \ge k+l+1\}.$

Accordingly we obtain  an admissible simplicial partition of $\R_+^2$ satisfying the flatness criterium of \autoref{flat}.
Then  let $g: \R_+\times\R_+\to \R_+$  be the unique continuous function which is affine on each $T_{kl}^\pm$ and such that  
\[ g(k,l) = \ell(0) (k^2 + l^2) + 2 \min \gra{\ell(0), \ell(|x-y|)} k l, \quad \forall (k,l) \in \N^2. \]
Note that the positivity of $q_\Sigma$ here means that $\ell(0) \ge \ell(|x-y|)$.

\begin{lemma} \label{m=2}
	The piecewise affine function $g$ defined above is convex.
\end{lemma}

\begin{proof} Setting $L_0 \eqdef \ell(0)$ and $L_1 \eqdef \min\gra{\ell(0), \ell(|x-y|)}$, the gradient of $g$ is piecewise constant and given on $Y_{kl}$ by:
\begin{equation} \label{jumps}
\nabla g = \begin{cases} 
L_0 (2k+1, 2l+1) + L_1(l,k) & \text{if $(k,l) \in T_{kl}^-$}\\
L_0 (2k+1, 2l+1) + L_1 (l+1,k+1) & \text{if $(k,l) \in T_{kl}^+$.}
\end{cases}
\end{equation}

The distributional Hessian of $g$ will be a rank tensor measure concentrated on the jump set of $\nabla g$. Checking the convexity of $g$ reduces then to check that $\nabla^2 g \ge 0$ which means that the normal jumps of $\nabla g$ are non-negative
along the sides of each triangle (or equivalently that the distributional Laplacian of $g$ is non-negative). 

In view of \eqref{jumps}, an easy computation shows that that the normal jumps along the horizontal or vertical sides are all equal $2 (L_0 - L_1)$, while the normal jumps along the oblique sides are equal to $2 \sqrt{2}\, L_1$. Accordingly, the distributional Laplacian  $\Delta g$ is a non-negative periodic measure.
\end{proof}

As a direct consequence of \autoref{explicit}, \autoref{partition} and \autoref{m=2}, we get the following.

%

\begin{corollary} \label{two-dirac} Let $\Sigma=\gra{x,y}$. Then for every $s,t \in [0,1]$ such that $s+t\le 1$, we have  
 \[ \ov{C_N}(s \delta_x + t \delta_y) = \frac{f_\Sigma(Ns, Nt)}{N(N-1)}. \]
 
 In particular, if $s + t = \frac{K}{N}$ with $K\in\N$, then 
 \[ \ov{C_N}(s \delta_x + t \delta_y) = C_K \pa{\frac{s}{s+t} \delta_x + \frac{t}{s+t}\delta_y}. \]
\end{corollary}

\begin{remark} The case of a single Dirac mass $\Sigma=\{x\}$ is recovered by taking $t=0$. Then $h(s) \eqdef \ov{C_N}(s \delta_x)$ is the monotone continuous function on $[0,1]$ such that $h(\frac{k}{N}) = \ell(0)\, \frac{k(k-1)}{N(N-1}$ and $f$ is affine on each interval $[\frac{k}{N}, \frac{k+1}{N}].$

In the case of two Dirac masses at $x, y$ such that $\ell(|x-y|) < \ell(0)$, an easy computation gives $f_{\Sigma}(k,l) = h(\frac{k+l}{N})$ and we reduce to the single Dirac case since $\ov{C_N}(s \delta_x + t \delta_y) = \ov{C_N}((s+t) \delta_x)$.  \end{remark}

\section{Duality and \texorpdfstring{$\Gamma$}{Gamma}-convergence}  \label{sec:limits}

\subsection{Duality}

In the following  the space of bounded Borel measures $\Meas(\R^d)$ is identified as the dual of $C_0(\R^d)$ and
we see $C_N$ as a convex functional defined on the whole space $\Meas(\R^d)$ by setting $C_N(\rho)=+\infty$ if $\rho\notin \Prob(\R^d)$.

By a classical result (see (e) in\eqref{prop-Fenchel}), the relaxed functional $\ov{C_N}$  can be recovered as the Fenchel biconjugate of $C_N$, that is:
\begin{equation}\label{dualrelax}
\overline{C_N}(\rho) = M_N^*(\rho) = \sup_{v\in C_0(\R^d)} \gra{\bra{v, \rho} - M_N(v)} .
\end{equation}
where  $M_N \colon C_0(\R^d) \to [-\infty,+\infty]$ the Fenchel conjugate of $C_N$, reeds
$$ M_N(v) \eqdef C_N^*(v) = \sup_{\rho \in \Prob(\R^d)} \gra{\bra{v, \rho} - C_N(\rho)}. $$
The existence of an optimal Lipschitz  potential $v$ in \eqref{dualrelax} has been proved when $\ell$ is a Coulomb type interaction
under mild assumptions (\emph{cf.} \cite{buttazzo2018continuity} for the case $\|\rho\| = 1$ and \cite{bouchitte2020relaxed} for the case $\|\rho\| < 1$).
 
A characterization of $M_N$ is available in terms of a maximization problem over configurations of $N$ points in $\R^d$.
Recalling the notation $S_N v(x)= \frac{1}{N} \sum_{i=1}^N v(x_i)$ for $x=(x_1, \dotsc, x_N)$,  we have the following result:
\begin{lemma}\label{MN+}
	For every $v \in C_0(\R^d)$,
	\begin{equation} \label{eq:MNdefinition}
	M_N(v) = \sup \gra{S_N v(x) - c_N(x) \st x \in (\R^d)^N}.
	\end{equation}
	It follows that $M_N(v) = M_N(v_+) \ge 0$. In addition, we have:
	\begin{equation} \label{MNbounds}
		\frac{1}{N} \sup v_+ \le M_N(v) \le \sup v_+.
	\end{equation}
\end{lemma}

\begin{remark} \label{MN<}
	If $\ell(0) > 0$, we have the strict inequality $M_N(v) < \sup v_+$  for any $v$ whose positive maximum is attained at a unique point
	$\bar x$. Indeed the  equality $M_N(v) = \sup v_+$  would imply that, at some point $x=(x_1,\cdots, x_N)$,  we have  $M_N(v)(x) = \max v_+$ while $c_N(x) = 0$,  which is not possible since all $x_i$ would coincide with $\bar x$ so that $c_N(x)=\ell(0)$.
\end{remark}

\begin{proof}
	Given $x_1, \dotsc, x_N \in \R^d$, let
	\[ \rho = \frac{1}{N} \sum_{j = 1}^N \delta_{x_j} \quad \text{and} \quad P = \frac{1}{N!} \sum_{\sigma \in \mathfrak{S}_N} \delta_{x_{\sigma(1)}} \otimes \dotsb \otimes \delta_{x_{\sigma(N)}}. \]
	Observe that $\rho \in \Prob(\R^d)$, and $P \in \Pi_N(\rho)$. Hence
	\begin{ieee*}{rCl}
		S_N v(x_1, \dotsc, x_N) - c_N(x_1, \dotsc, x_N) & = & \int v d\rho - \int c_N dP \\
		& \leq & \int v d\rho - C_N(\rho) \leq M_N(v),
	\end{ieee*}
	which gives an inequality.
	On the other hand, for any $\rho \in \Prob(\R^d)$, if $P \in \Pi_N(\rho)$ is optimal in \eqref{eq:OTproblem} one has
	\[ \int v d\rho - C_N(\rho) = \int (S_N v - c_N) dP \leq \sup (S_N v - c_N) \]
	Passing to the supremum on the left-hand side one gets the converse inequality whence  \eqref{eq:MNdefinition}.
	
	The fact that $M_N(v)\ge 0$ is then straightforward by sending all $x_k$ to infinity in such a way that $c_N(x_1, \dots, x_N)\to 0$.
	
	On the other hand, we have obviously $M_N(v)\le M_N(v_+)$. To show the opposite inequality let us fix $\e>0$ and take $(x_1, \dotsc, x_N) \in (\R^d)^N$ such that
	\[ S_N v_+(x_1, \dotsc, x_N) - c_N(x_1, \dotsc, x_N)\ge M_N(v_+) -\e. \]
	
	Then sending to infinity every $x_k$ such that $v(x_k) < 0$ while keeping the other $x_j$'s fixed will increase the value of $S_N(v)$ to $S_N(v_+)$ and decrease the value of $c_N(x_1, \dotsc, x_N)$ (since all terms of the kind $\ell(\abs{x_j-x_k})$ will vanish). 
	Accordingly we get $M_N(v) \ge S_N(v_+)- c_N(x_1, \dotsc, x_N)\ge M_N(v_+) -\e$, hence the desired inequality since $\e$ is arbitrary.
	
	To show \eqref{MNbounds}, we may now assume that $v\ge 0$. Then the upperbound $M_N(v)\ge \sup v$ is trivial. On the other hand, by selecting $x_1$ such that $v(x_1) = \max v$ and by sending the other $x_j$'s to infinity, we conclude with the lower bound $M_N(v) \ge \frac{1}{N} v(x_1)= \frac{1}{N} \sup v.$
\end{proof}

In the same spirit as of \autoref{MN+}, we have an alternative lower bound inequality for $M_N(v)$ namely:
\begin{lemma} \label{MN-MK-inequality}
For every $2 \leq K \leq N$ one has
	\begin{equation} \label{MNMK} M_N(v) \geq \frac{K(K-1)}{N(N-1)} M_K \left( \frac{N-1}{K-1} v \right). \end{equation}
\end{lemma}

\begin{proof}
	Let $\{x_1, \dotsc, x_K\}$ be an optimal $K$-points configuration up to a small $\e>0$ for $M_K\left( \frac{N-1}{K-1} v \right)$ as defined in \eqref{eq:MNdefinition}. Then we may complete to obtain a $N$-points configuration by adding $x_{K+1}, \dotsc, x_N$ tending to $\infty$ so that, in view of $(H3)$, we have: 
	\begin{ieee*}{+rCl+x*}
		M_N(v) & \geq & \frac{1}{N} \sum_{j = 1}^K v(x_j) - \frac{2}{N(N-1)} \sum_{0 \leq i < j \leq K} \ell(\abs{x_i-x_j}) \\
		& = & \frac{K(K-1)}{N(N-1)} \pa{ \frac{N-1}{K(K-1)} \sum_{j = 1}^K v(x_j) - \frac{2}{K(K-1)} \sum_{0 \leq i < j \leq K} \ell(\abs{x_i-x_j})} \\
		& = & \frac{K(K-1)}{N(N-1)} \pa{  \frac{1}{K} \sum_{j = 1}^K \frac{N-1}{K-1} v(x_j) - c_K(x_1, \dotsc, x_K) } \\
		& \ge  & \frac{K(K-1)}{N(N-1)} \left( M_K \pa{ \frac{N-1}{K-1} v } -\e\right) ,
	\end{ieee*}
	hence \eqref{MNMK} by sending $\e\to 0$.  
\end{proof}
\medskip
	
In view of the equality $M_N(v)=M_N(v_+)$ (see \autoref{MN+}), the supremum in \eqref{dualrelax} can be restricted to non negative $v$ and rewritten in the alternative form
\begin{equation}\label{alternative}
\overline{C_N}(\rho) = \sup_{\lambda\ge0, v\in C_0^+(\R^d)} \gra{ \int v d\rho - \lambda \st S_N(v) \le c_N + \lambda \quad \text{in $(\R^d)^N$}} 
\end{equation}
from which we deduce  the following monotonocity property: 

\begin{lemma}\label{mono+add} Let $\rho, \nu \in  \Prob_-(\R^d)$ such that $\rho\le \nu$. Then $\ov{C_N}(\rho) \le \ov{C_N}(\nu).$
\end{lemma}

\begin{proof} 
Since $ \int v d\rho \ge \int v d\rho$ for every $v\in C_0^+(\R^d)$, the inequality $\ov{C_N}(\rho) \le \ov{C_N}(\nu)$ follows from
\eqref{alternative}.
\end{proof}



\begin{remark} It turns out that, by applying \cite[Proposition 3.8]{bouchitte2020relaxed}, the pointwise inequality constraint appearing in \eqref{alternative} can be drastically reduced in practice and replaced by the same condition $S_N(v)\le c_N + \lambda$ holding $\tilde{\rho}^{\otimes N}$ a.e. on $X^d$, where $X=\R^d\cup\{\omega\}$ is the compactified space introduced in Section \ref{2formulae} (with $c_N$ and $S_N(v)$ being extended to $X^d$). 

Accordingly, for a discrete measure $\rho = \sum_{i=1}^m t_i \delta_{x_i}$, we are led to a linear programming problem involving $m+1$ unknowns in $\R_+$, namely $v_i = v(x_i)$ ($1\le i\le m$) and $v_{m+1}=\lambda$.
For instance, in the case of a two Dirac masses measure as studied in Subsection \ref{Dirac}, we are led to:
\[ \ov{C_N}(s \delta_x + t \delta_y) = \sup \gra{s v_1 + t v_2 + v_3}, \] 
subject to the following constraints holding for every $k\ge 0, l\ge 0$ with $k+l\le N$:
\[ \frac{k}{N} v_1 +\frac{l}{N} v_2 + \pa{1 - \frac{k+l}{N}} v_3 \le \frac{k(k-1) + l(l-1)}{N(N-1)} \ell(0) + \frac{2kl}{N(N-1)} \ell(|x-y|) \]
\end{remark}

\subsection{Pointwise and \texorpdfstring{$\Gamma$}{Gamma}-convergence}

The asymptotic behavior of  the functionals  $C_N$ and $M_N$ as $N\to\infty$ will be a direct consequence of the two following results.
   
\begin{lemma}[monotonicity] \label{monotonicity}
	The sequences $(C_N)_{N\geq 1}$ and $(\ov{C_N})_{N\geq 1}$ are monotone non-decreasing. The sequence $(M_N)_{N\geq 1}$ is monotone non-increasing.
\end{lemma}

Note that, in order to get the monotonicity property of $C_N$ with respect to the number $N$ of marginals, the normalization factor $\frac{1}{N(N-1)}$ used in the definition \eqref{def:c_N} is essential.

\begin{proof}
	We observe that
	\[ C_N (\rho)= \inf_{\sigma \in \Pi_{2,N}(\rho)} \int \ell(\abs{x_1-x_2}) d\sigma(x_1,x_2), \]
	where $\Pi_{2,N}(\rho)$ denotes the set of 2-body marginals of measures in $\Pi_N(\rho)$. This $N$-representability constraint becomes more restrictive as $N$ increases, whence the claimed monotonicity property of $(C_N)$, hence also for $\ov{C_N}$. By passing to the Fenchel conjugates, it follows that
	the sequence $(M_N)$ is non increasing.
\end{proof}

\begin{lemma} \label{equiLip}
	The family $\{M_N, N\geq 2\}$ is equi-Lipschitz, with Lipschitz constant equal to 1.
\end{lemma}

\begin{proof}
	Let $v_1, v_2 \in C_0(\R^d)$. Take $x_1, \dotsc, x_N \in \R^d$ optimal up to a threshold $\epsilon$ for $M_N(v_1)$ in \eqref{eq:MNdefinition}. Then
	\[ M_N(v_1) - M_N(v_2) \leq \frac{1}{N} \sum_{j = 1}^N \left( v_1(x_j) - v_2(x_j) \right) + \epsilon \leq \norm{v_1 - v_2}_{\infty} + \epsilon. \]
	By letting $\epsilon \to 0$, and then switching the roles of $v_1$ and $v_2$ we get the thesis.
\end{proof}

From \autoref{monotonicity}, we infer the existence of pointwise limits for the functionals $M_N$ and $\overline{C_N}$. In the following we will denote
\begin{equation}\label{def:MN-CN}
M_{\infty}(v) \eqdef \lim_{N \to \infty} M_N(v), \quad  C_{\infty}(\rho) \eqdef \lim_{N \to \infty} \overline{C_N}(\rho).
\end{equation}

Clearly $C_\infty$ defines a convex lower semicontinous functional on $\Meas(\R^d)$ whose domain is a subset of $\Prob_-(\R^d)$.  
Besides  $M_\infty$ enjoys the same properties as $M_N$, namely  to be a convex $1$-Lipschitz continuous functional on $C_0(\R^d)$  depending only of the positive part of its argument.


\begin{lemma} \label{CIMIduality}
We have the duality relations
$$ C_\infty = M_\infty^*, \quad M_\infty = C_\infty^*.$$
Moreover, if $\|\rho\| = 1$, then $C_\infty(\rho) = \lim_{N \to \infty} C_N(\rho)$.
\end{lemma}
	
\begin{proof} As $\overline{C_N}= M_N^*$, we have $C_\infty = \sup_N M_N^*$ while $M_\infty = \inf_N M_N$.
Then the first equality follows from the general identity $(\inf_N M_N)^* = \sup_N M_N^*$. As $M_\infty$ is continuous and convex, we deduce that $C_\infty^* = (M_\infty)^{**}= M_\infty$.
The last statement follows from \eqref{def:MN-CN} since $\ov{C_N}(\rho)= C_N(\rho)$ for $\|\rho\|=1$.
\end{proof}

As a consequence of the above results and of some classical results in $\Gamma$-convergence theory for convex functionals, we
deduce:

%
%


\begin{thm} \label{Gconvergence}
 The functionals $M_N$ $\Gamma$-converge to  $M_\infty$ while both functionals $C_N$ and $\overline{C_N}$ are $\Gamma$-converging to  $C_\infty$.	
\end{thm}

\begin{proof} The first statement follows by applying \autoref{prop:pointwise-gamma} with $G_N = M_N$ and $X = C_0(\R^d)$, while the second one follows by applying \autoref{prop:attouch-dual} to	the sequence $(F_N)$ defined on 
$X^*= \Meas(\R^d)$ by setting  $F_N(\rho)= C_N(\rho)$ if $\rho\in \Prob(\R^d)$ and $F_N(\rho)=+\infty$ otherwise.
\end{proof}

\subsection{More properties of \texorpdfstring{$C_\infty$}{Cinfty} and \texorpdfstring{$M_\infty$}{Minfty}}

First of all we observe that the condition $(H4)$
ensures that, for every $r>0$, one has 
\begin{equation} \label{def:Krd}
K(r,d) \eqdef \frac1{(\omega_d r^d)^2}\int \One_{B(0,r)}(x_1) \One_{B(0,r)}(x_2) \ell(\abs{x_1-x_2}) dx_1 dx_2 < +\infty.
\end{equation}
\begin{lemma} \label{CI-dense-domain}
	Assume that $\ell$ satisfies the local integrability assumption (H4). Then the domain of $C_\infty$ is a dense subset of $\Prob_-(\R^d)$. As a consequence, we have
	\[ \lim_{t\to +\infty} \frac{M_\infty(tv)}{t} = \sup v_+ \quad \forall v\in C_0(\R^d). \]
\end{lemma}

\begin{proof}\ By construction the domain of the relaxed functional $\ov{C_N}$ is a subset of $\Prob_-(\R^d)$ and
therefore $C_\infty(\rho)= \sup_N \ov{C_N} (\rho) = +\infty$ if $\rho\notin \Prob_-(\R^d)$. It follows that the closure of $\{\rho \st C_\infty(\rho) < +\infty\}$ is a  weakly* compact convex subset $ \mathcal{K}\subset\Prob_-(\R^d)$.  
Thus the desired density property holds if we can show that $\delta_x \in \mathcal{K}$ for every $x\in\R^d$. Let us consider for every $r>0$ the uniform probability $\rho_r$ on the ball $B(x,r)$ and  the transport plan $P_{N,r} = \underset{N\text{ times}}{\underbrace{\rho_r \otimes \dotsb \otimes \rho_r}}$. We get
\[ C_N(\rho_r) \leq \int c(x_1, \dotsc, x_N) dP_{N,r} = \int \ell(\abs{x_1-x_2}) d(\rho_r \otimes \rho_r)  = K(r,d) < +\infty, \]
with $K(r,d)$ given by \eqref{def:Krd}. Thus $\rho_r \in \mathcal{K}$ while $\rho_r \wconv \delta_x$ as $r\to 0$.
Summarizing we have proved the equality $\mathcal{K}=\Prob_-(\R^d)$. Passing to the support functions, we recover the recession function of $M_\infty$ by applying  \eqref{recession} (with $f=M_\infty$ and $X= C_0(\R^d)$. We are led to:
$$  \lim_{t\to +\infty} \frac{M_\infty(tv)}{t} = \sup_{\rho\in \mathcal{K}}  \bra{v,\rho}  = \sup_{\rho\in \Prob_-(\R^d)}  \bra{v,\rho}= \sup v_+ .$$
  \end{proof}
\begin{lemma} \label{slopeM}
	Assume that $\ell > 0$ on $[0,+\infty)$. Then 
	\[ \lim_{t\to 0^+} \frac{M_\infty(t v)}{t} = 0 \quad \forall v\in C_0(\R^d). \]
\end{lemma}
\begin{proof}\ Since the map  $t\mapsto H_t(v):= \frac{M_\infty(t \varphi)}{t}$ is monotone non decreasing, we have 
$$ H_0(v):=\lim_{t\to 0^+}  \frac{M_\infty(t v)}{t} =\inf_{t>0}  H_t(v) .$$
Clearly the function $H_0$ is convex and Lipschitz on $C_0(\R^d)$. Therefore $H_0$ coincides with its biconjugate $H_0^{**}$.
 An easy computation shows that  $H_0^*  = (\inf_{t>0}  H_t)^* = \sup_{t>0} H_t^*$
while 
 $$H_t^*(\rho)= \sup \left\{ \frac1{t}\,  (\bra{\rho, tv} - M_\infty(tv)):  v\in C_0(\R^d)\right\} = \frac{C_\infty(\rho)}{t},$$ 
 for every $\rho\in \Prob_-(\R^d)$ and $t>0$.  It follows that $H_0^*$ coincides with the indicator function of the 
subset $\{\rho : C_\infty(\rho)\le 0\}$ that is of $\{\rho=0\}$ ( since $\ell>0$). 
 Therefore we have $H_0= H_0^{**} \equiv 0$ as claimed. 
 \end{proof}

\begin{remark}\label{slopes} Under the assumptions of \autoref{CI-dense-domain} and  \autoref{slopeM}, we see that, for every $v\in C_0(\R^d)$, the monotone function $ t \in [0,+\infty] \to \frac{M_\infty(t v)}{t}$ is increasing from $0$ to $\max(v_+)$.
Moreover, if $\sup v>0$, it is strictly increasing since the constancy on some interval would imply that $t\mapsto M_\infty(tv)$ is affine on an interval starting from $0$ in contradiction with a vanishing derivative at $t=0$.
 
On the other hand, we notice that the condition $\ell(0) > 0$ is sufficient to ensure that $C_\infty$ is not identically zero on $\Prob_-(\R^d)$.
Indeed otherwise, we would have  $M_N(v)\ge M_\infty(v)= (C^\infty)^*(v)= \sup v_+$ for every $v\in C_0(\R^d)$ and $N\ge 2$, whence a contradiction with the strict inequality pointed out in \autoref{MN<}.
\end{remark}

\begin{lemma} \label{subquadra} \ The convex  functional $M_\infty$ is 1-Lipschitz on $C_0(\R^d)$ and, for every $v\in C_0(\R^d)$,
the map $ \la\in \R_+ \mapsto \dfrac{M_\infty(\la v)}{\la^2}$ is monotone non increasing.
In particular, under $(H4)$, we have $M_\infty(v)>0$ for every $v\in C_0(\R^d)$ such that $\sup v>0.$
\end{lemma}

\begin{proof}
     Let $t\ge 1$, $w\in C_0(\R^d)$  and choose a sequence $(k_N)_{N \geq 2}$ such that $\lim_{N\to\infty} \frac{k_N}{N} = t^{-1}$. Then, by  applying the inequality \eqref{MNMK}, we get:  
\[ M_N(v) \geq \frac{k_N(k_N-1)}{N(N-1)} M_{k_N} \pa{\frac{N-1}{k_N-1} w} \ge \frac{k_N(k_N-1)}{N(N-1)} M_\infty \pa{\frac{N-1}{k_N-1} w}, \]
 thus , after sending $N\to\infty$, the inequality 
\begin{equation} \label{subtwo}  M_\infty(t w) \ \le t^2\, M_\infty(w)  
\qquad  \forall w\in C_0(\R^d)\ ,\ \forall t\ge 1. \end{equation}
 Let now $\la,\mu$ such that $0<\la<\mu$. Then, by plugging  $t= \frac{\la}{\mu}$ and $ v= \mu w$ in \eqref{subtwo}, we deduce that
 $ \frac{M_\infty(\la v)}{\la^2} \le \frac{M_\infty(\mu v)}{\mu^2}, $   hence the desired monotonicity property.

 In order to show the last statement, we notice that the convexity of $M_\infty$ implies that the set of real numbers $\gra{t\ge 0 : M_\infty(tv)\le 0}$
 is a closed non empty interval starting from $0$. Then the inequality \eqref{subtwo} implies that this interval is either
  $\gra{0}$ or the half line $\R_+$. Under (H4), the second alternative is ruled out if  $\sup v>0$ since, by \autoref{CI-dense-domain}, the slope at infinity  $\lim_{t\to+\infty} \frac1{t} M_\infty(tv)$ must be positive.
\end{proof}

\begin{prop} \label{homogeneity-thm} 
 The convex functional $C_\infty$ satisfies the inequality: 
	\begin{equation} \label{subhomogeneous}  C_\infty(\theta \rho) \leq \theta^2 C_\infty(\rho) \quad \forall \rho \in \Prob(\R^d), \forall \theta \in [0,1]. 
	\end{equation}	
	Moreover, if \autoref{k-conjecture} holds true, then  $C_\infty$ is 2-homogeneous, \ie, the inequality \eqref{subhomogeneous} is an equality.
\end{prop}

\begin{proof} In order to show \eqref{subhomogeneous}, it is enough to consider a sequence $(k_N)_{N \geq 2}$ such that $\lim \frac{k_N}{N} = \theta$ and apply the inequality \eqref{trueineq}:
\[ \ov{C_N}\pa{ \frac{k_N}{N} \rho} \le \frac{k_N(k_N-1)}{N(N-1)} C_{k_N} (\rho). \]
Then, by exploiting the $\Gamma$-convergence of $\ov{C_N}$ to $C_\infty$, we get $\liminf_{N\to\infty} \ov{C_N} \pa{ \frac{k_N}{N} \rho} \ge C_\infty(\theta \rho)$ while, by the pointwise convergence $\ov{C_{k_N}} (\rho)\to C_\infty(\rho)$ (recall that $\ov{C_{k_N}} (\rho) =C_{k_N} (\rho) $ when $\|\rho\| = 1$), the right hand side converges to $\theta^2 \, C_\infty(\rho)$.

Next let us assume that \eqref{gapvanish} holds true and let $\rho_1, \dotsc, \rho_N$ and $a_1, \dotsc, a_N$ be optimal in \eqref{relaxed-CN} for $\theta \rho$, with $a_1 = \dotsb = a_{k-1} = 0$. Since $C_j\ge C_k$ for $j\ge k$, by the convexity of $C_k$, we get
\begin{ieee*}{rCl}
	\overline{C_N}(\theta \rho) & = & \sum_{j = k}^N a_j \frac{j(j-1)}{N(N-1)} C_j(\rho_j) \\
	& \geq & \frac{k-1}{N-1} \sum_{j = k}^N a_j \frac{j}{N} C_k(\rho_j) \geq \frac{k-1}{N-1} \theta C_k(\rho).
\end{ieee*}
The inequality above will hold in particular for $k= \underline{k}(N,\theta \rho)$. Therefore, 
by taking a sequence $(N_h)_{h \in \N}$ such that $\lim_{h \to \infty} \frac{\underline{k}(N_h,\theta \rho)}{N_h} = \theta$, 
we may conclude that $C_\infty(\theta \rho) \geq \theta^2 C_\infty(\rho)$, which is the converse inequality of \eqref{subhomogeneous}.
\end{proof}

%

\section{Relations with the direct energy} \label{sec:direct-energy}

In Section \ref{sec:limits}, we obtained a formal description of the limit functionals $M_\infty$ and $C_\infty$, namely that the $\Gamma$-limit $M_\infty$ coincides with the pointwise limit of the non increasing sequence $(M_N)$ while $C_\infty = \sup \ov{C_N}$ coincides with the Fenchel conjugate of $M_\infty$. 
Our aim now is to give a more explicit description of these functionals  relying on the fact that $M_\infty$ coincides with the Fenchel conjugate of the so called \emph{direct energy} $D \colon \Meas(\R^d) \to [0, +\infty]$ given by \eqref{def:D}
\begin{equation}\label{def:D}
D(\rho) \eqdef \begin{cases}
\displaystyle \int \ell(\abs{x-y}) d\rho(x) d\rho(y) & \text{if $\rho \in \Prob(\R^d)$} \\
+ \infty & \text{otherwise.}
\end{cases}
\end{equation}

The name ``direct energy'' is inherited from a physical model where $\rho$ represents a charge density, as this functional equals (up to constants) the potential energy due to the self-interaction of the density $\rho$.
It is convenient to introduce also the $2$-homogeneous extension of $D$ to sub-probabilities  $D_2 \colon \Meas(\R^d) \to [0, +\infty]$ \ie,
\begin{equation} \label{Dalpha}
    D_2(\rho) \eqdef \begin{cases}
	\|\rho\|^2 D \left( \dfrac{\rho}{\|\rho\|} \right) & \text{if $\rho \in \Prob_-(\R^d)$} \\
	+ \infty & \text{otherwise.}
	\end{cases} 
\end{equation}

By construction we have $D_2 \leq D$. Moreover, by observing  that  $D_2(\rho)=\bra{c_2, \rho \otimes \rho}$ for every $\rho \in \Prob_-(\R^d)$, it is easy to deduce from (H2)  that $D_2$ is weakly* lower semi-continuous as a functional on $\Meas(\R^d)$. In contrast the functional $D$, whose domain is not closed, is not lower semi-continuous. 
Therefore, we will consider the  weak* lower semi-continuous envelope  of $D$ defined for every $\rho \in \Prob_-(\R^d)$ by
 \begin{equation}\label{relaxD}
\ov{D}(\rho) = \inf_{\rho_n \wconv \rho} \liminf_{n \to \infty} D(\rho_n). 
\end{equation}

\begin{lemma} \label{D2-D}
Let $D$ and $D_2$ be defined by \eqref{def:D},\eqref{Dalpha} respectively. Then it holds
\begin{equation}\label{D*=D2*}  \ov{D} \leq D_2\quad,\quad 
D^* = D_2^* .	
\end{equation}  
\end{lemma}

 \begin{proof}  In view of the property (d) in \autoref{A}, we have $D^{*} = (\overline{D})^{*}$. Thus the first inequality of 
 \eqref{D*=D2*} will imply that  $D_2^* \le D^*$, hence $D_2^*=D^*$ since the converse inequality follows  trivially from $D_2\le D$.
 So it is enough to show that $\overline{D} \leq D_2$ which is equivalent to the following
 claim:
 \begin{equation} \label{approx}
\forall \rho\in \Prob_-(\R^d) \quad \exists \rho_n \wconv \rho \colon \limsup_{n\to\infty}  D(\rho_n) \leq D_2(\rho).
\end{equation}

First we show \eqref{approx} when $\rho \in \Prob_-(\R^d)$ has a compact support. Without loss of generality we may assume that $\|\rho\|\ge 0$ and $D_2(\rho) < +\infty$. For $h \in \R^d$, we denote by $\tau_{h} \rho$ the push-forward of $\rho$ through the translation map $x\mapsto x+h$. We will exploit several times the following property 
\begin{equation}\label{translation}
\lim_{\abs{h} \to \infty} \bra{c_2, \rho \otimes \tau_{h} \rho} = 0.
\end{equation}
To check \eqref{translation}, we simply notice that if $\supp \rho\subset B_R$, then for $\abs{h} \geq 3R$ we have
\[ \bra{c_2, \rho \otimes \tau_{h} \rho} = \iint_{B_R \times B_R} \ell(\abs{x-y-h}) d\rho(x) d\rho(y)
 \leq \sup \gra{\ell(r) \colon r\ge R}, \]
whence $\limsup_{\abs{h} \to \infty} \bra{c_2, \rho \otimes \tau_{h} \rho} \leq \limsup_{r\to +\infty} \ell(r)=0,$  thanks to (H3).

Accordingly, given a unitary direction $e \in \R^d$ , we can choose an increasing sequence of positive reals $(R_n)$ such that $R_n\to \infty$ and 
\begin{equation} \label{cross-term}
\bra{c_2, \rho \otimes \tau_{r e} \rho} \leq \frac{1}{n^2} \quad \forall r \geq R_n 
\end{equation}

For given $n\in N^*$, we define $n$ collinear vectors $h_{n,j} = R_n j e$ for $j = 1, \dotsc, n$, so that $\abs{h_{n,i} - h_{n,j}} \geq R_n e$ for every $i \neq j$. Then we set:
\[ \rho_n = \rho + \frac{1 - \|\rho\|}{n \|\rho\|} \sum_{j = 1}^n \tau_{h_{n,j}} \rho. \]
By construction $\rho_n \in \Prob(\R^d)$. Moreover, since $|h_{n,j}| \geq R_n \to \infty$, $\rho_n$ agrees with $\rho$ on any compact subset as soon $n$ is large enough, hence we have $\rho_n \wconv \rho$. On the other hand, by the invariance of $D_2$ under translations and the choice of the $h_{n,j}$'s, we get:
\begin{align*}
	D(\rho_n) & = D_2(\rho) + \frac{2(1 - \|\rho\|)}{n \|\rho\|} \sum_{\substack{i,j = 0 \\ i \neq j}}^n \bra{c_2, \tau_{h_{n,i}} \rho \otimes \tau_{h_{n,j}} \rho} +  \frac{(1 - \|\rho\|)^2}{n^2 \|\rho\|^2 } \sum_{j = 1}^n\,  D_2(\tau_{h_{n,j}} \rho ) \\
	& = D_2(\rho) + \frac{2(1 - \|\rho\|)}{n \|\rho\|} \sum_{\substack{i,j = 0 \\ i \neq j}}^n \bra{c_2, \rho \otimes \tau_{(h_{n,i} - h_{n,j})} \rho} + \frac{(1 - \|\rho\|)^2}{n^2 \|\rho\|^2} \sum_{j = 1}^n D_2(\rho) \\
	& \leq D_2(\rho) + \frac{2(1 - \|\rho\|)}{n \|\rho\|} + \frac{(1 - \|\rho\|)^2}{n  \|\rho\|^2} D_2(\rho),
\end{align*}
where,  in order to derive the last inequality, we used the fact that, by \eqref{cross-term}, we have $\bra{c_2, \rho \otimes \tau_{(h_{n,i} - h_{n,j})} \rho}\le \frac1{n^2}.$
Summarizing, we found a recovering sequence $(\rho_n)$ satisfying  \eqref{approx} and therefore $\ov{D}(\rho) \leq D_2(\rho)$.

To extend this inequality to non compactly supported  $\rho \in \Prob_-(\R^d)$, it is enough to consider, for every $n \in \N^*$, the measure $\rho_n \eqdef \rho \restr_{B(0,n)}$ which clearly satisfies $\rho_n \wconv \rho$ as $n\to\infty$. Then, since $D_2(\rho_n) \leq D_2(\rho)$,
we have
\[ \overline{D}(\rho) \leq \liminf_{n \to \infty} \overline{D}(\rho_n) \leq \liminf_{n \to \infty} D_2(\rho_n) \leq D_2(\rho). \]  
This concludes the proof of \eqref{D*=D2*}.
\end{proof}

%
%

\begin{thm} \label{MIformula}
	 Let $\ell$ satisfy the standing assumptions. Then 
	
	\med
 $(i)$\ It holds\ $M_\infty= D^*$, hence  for every $v \in C_0(\R^d)$: 
	\begin{equation}\label{Minfty=sup}
		M_\infty(v) = \sup_{\rho \in \Prob(\R^d)} \int \left(\frac{v(x)+ v(y)}{2} - \ell(\abs{x-y}) \right)\,  d\rho(x) d\rho(y).
	\end{equation}
	Moreover, if the cost $\ell$ is bounded, the following estimate holds for $N\ge 1$:
	\begin{equation}\label{estiMinfty}
	M_\infty(v) \le M_N(v) \le M_\infty(v) + \frac{1}{N} (\sup \ell + \sup v_+).
	\end{equation}
	
	\med
	$(ii)$\ 
 We have he following duality relations:
	\begin{equation}\label{dualD}
      C_\infty = D^{**}= D_2^{**}.
    \end{equation}
    In particular, if $D$ is convex, then $C_\infty= \ov{D}$ while $C_\infty(\rho)=D(\rho)$ if $\|\rho\| = 1$. If moreover $D_2$ is convex, then\  $C_\infty = D_2.$

\end{thm}

\begin{proof} It is immediate to check that the right hand side supremum in \eqref{Minfty=sup} coincides with $D^*(v)$. Let us prove that $D^*\le M_\infty$. Due to the pairwise-interaction structure of the cost $c_N$ and recalling \eqref{eq:MNdefinition}, we have for every $\rho \in \Prob(\R^d)$ and $N \geq 2$
\[ \bra{v,\rho} - D(\rho) = \bra{S_2 v - c_2, \rho \otimes \rho}
= \langle S_N v - c_N, \underset{N\text{ times}}{\underbrace{\rho \otimes \dotsb \otimes \rho}} \rangle 
 \leq M_N(v).
\]

The inequality  $D^*(v) \leq M_\infty (v)$ follows by taking the supremum with respect to $\rho\in \Prob(\R^d)$ on the left-hand side and the infimum in $N$ on the left-hand side.

Let us prove now the converse inequality $ M_\infty \leq D^*$.
 To that aim, we fix  $\e >0$ and, for every $N\ge 2$, we pick $x^{(N)} =(x_1^{(N)}, x_2^{(N)}, \dotsc, x_N^{(N)}) \in \Rdn$  such that
\[ M_N(v) \leq (S_Nv - c_N)(x^{(N)}) + \e. \]
Then we consider the symmetric transport plan $P_N = \frac{1}{N!} \sum_{\sigma \in \mathfrak{S}_N} \delta_{x_{\sigma(1)}^{(1)}} \otimes \dotsb \otimes \delta_{x_{\sigma(N)}^{(N)}}$ and denote by $\gamma_N \in \Prob(\R^d\times\R^d)$ the 
the 2-marginal projection of $P_N$. Note that $P_N \in \Pi_N(\rho)$ being $\rho = \frac{1}{N} \sum_{j = 1}^N \delta_{x_j^{(N)}}$, and
\begin{equation}\label{majMN}
 M_N(v) -\e \leq \bra{S_N v - c_N, P_N} = \bra{S_2 v - c_2, \gamma_N}.
\end{equation}
Next we approximate  $\gamma_N$ by $\widetilde{\gamma_N}\in \Prob(\R^d\times\R^d)$ given in the form:
\begin{equation}\label{intnuN}
\widetilde{\gamma_N} (A) = \int_A (Q \otimes Q) d p_N(Q), \quad \text{for every Borel subset $A \subset \R^d\times\R^d$,}  
\end{equation} 
where $p_N$ is suitably chosen element of  $\Prob(\Prob(\R^d))$.

Owing to a classical result by Diaconis and Freedman \cite[Theorem 13]{diaconis1980finite} (see the last statement in the proof therein), we may choose $p_N$ so that the gap $\mu_N:=\gamma_N - \widetilde{\gamma_N}$  is a balanced measure such that:   
\begin{equation}\label{Diaco}
 \int_{\R^d\times\R^d} (\mu_N)_+ = \int_{\R^d\times\R^d}  (\mu_N)_- \leq \frac1{N}.
\end{equation}
In view of \eqref{intnuN}, we have
\begin{equation*} \bra{S_2 v - c_2,\widetilde{\gamma_N}} \leq \sup_{Q\in \Prob(\R^d)} \bra{S_2 v - c_2, Q \otimes Q} = D^*(v). \end{equation*}
Therefore the estimate \eqref{majMN} implies that:	
\begin{equation} \label{gap-estimate}
M_N(v) -\e \leq D^*(v) + \bra{S_2v - c_2, \mu_N}.
\end{equation}
In order to conclude that $M_\infty\leq D^*$, we proceed now in two steps. 

\med
{\em Step 1:  we assume that $\sup \ell < +\infty$ and prove \eqref{estiMinfty}}. 
In this case, we have $ \bra{S_2v - c_2, \mu_N} \to 0$ as $N\to +\infty$ since $S_2v - c_2$ is bounded and $\mu_N\to 0$ 
in total variation owing to \eqref{Diaco}. Thus, from \eqref{gap-estimate}, we infer that $M_\infty(v) \le D^*(v)$ by sending $N\to\infty$ and then $\e\to 0$. Therefore we have $M_\infty=D^*$ and \eqref{Minfty=sup}
Next, keeping in mind that $ M_N(v)=M_N(v^+)$ and $M_\infty(v)=M_\infty(v_+)$,
it is not restrictive to assume that $v\ge 0$. Then by \eqref{gap-estimate} and \eqref{Minfty=sup}, we get:
\[ M_N(v) - \e \leq  M_\infty(v) +\pa{\int (\mu_N)_+} \sup v + \pa{\int (\mu_N)_-} \sup \ell  \]
whence the estimate  \eqref{estiMinfty} thanks to \eqref{Diaco} and since $\e$ is arbitrary small.

\med
{\em Step 2:  we remove the assumption $\sup \ell <+\infty$ and show that $M_\infty\le D_2^*$}. 
We use the truncated cost $\ell_h (r) = \ell(r) \wedge h$ being $h$ a large positive parameter and define accordingly
$D_{2,h} (\rho) \eqdef \bra{c_{2,h} , \rho\otimes\rho}$ for $\rho\in \Prob_-(\R^d)$, $c_{2,h}(x,y) \eqdef \ell_h(|x-y|)$ 
and $M_{N,h} (v) = \sup \{ S_N v - c_{N,h}\}$ for $v\in C_0(\R^d)$.    
Clearly it holds $M_N(v)\le M_{N,h} (v)$ so that $M_\infty(v) \le \lim_{N\searrow\infty} M_{N,h} (v)$.
By applying Step 1 to the bounded cost $\ell_h$ (which satisfies the standing assumptions), we deduce that $ M_\infty(v) \le D_{2,h}^* (v)$ for every $v\in C_0(\R^d)$ and $h>0$. Therefore, recalling that $D_2^*= D^*$ by \autoref{D2-D}, we are done if we show that 
 \begin{equation}\label{D2h*}
\limsup_{h\to +\infty}  D_{2,h}^* (v) \leq D_2^*(v).
\end{equation}
By the definition of the Fenchel conjugate, there exists $\rho_h \in \Prob_-(\R^d)$ such that
\[ \limsup_{h\to +\infty}  D_{2,h}^* (v) \le \limsup_{h\to +\infty} \ \bra{v, \rho_h} - D_{2,h} (\rho_h). \]
Without loss of generality, we may assume that there exists $\rho\in \Prob_-(\R^d)$ such that $\rho_h \wconv \rho$. 
The wished inequality \eqref{D2h*} follows once we can check that 
 \begin{equation}\label{lowerDh} \liminf_{h\to +\infty}  D_{2,h}(\rho_h)\ge D_2(\rho).\end{equation}
 
Let us fix $h_0>0$. Then, since  $D_{2,h} \ge D_{2,h_0}$ for every $h\ge h_0$, by the lower semicontinuity of $D_{2,h_0}$ we have
\[ \liminf_{h\to +\infty}  D_{2,h}(\rho_h)\ \ge\ \liminf_{h\to +\infty}  D_{2,h_0}(\rho_h)\ \ge\  D_{2,h_0}(\rho). \]
As $h_0$ is arbitrary , the claim \eqref{lowerDh}  follows since by monotone convergence:
\[ \lim_{h_0\nearrow +\infty}  D_{2,h_0}(\rho) = \int \sup_{h_0} \ell_{h_0}(|x-y|)\ \rho\otimes\rho =   \int  \ell(|x-y|)\ \rho\otimes\rho
  = D_2(\rho). \]
The proof of Step 2 and of the assertion (i) of \autoref{MIformula} is now complete.


\med
(ii)  A consequence of the assertion (i) above and of \autoref{CIMIduality} is that $C_\infty= M_\infty^*= D^{**}$, hence \eqref{dualD} since $D^*=D_2^*$.
%

Assume now that $D$ is convex. Then $C_\infty = \overline{D}$ in virtue of the property $(e)$ in \autoref{prop-Fenchel}. Since $D_2\le D$ while $D_2$ is lower semi-continous, we have $D_2\le \ov{D}\le D$. Noticing that  $D_2=D$ on $\Prob(\R^d)$ by construction,  we deduce that  $C_\infty(\rho)= \ov{D}(\rho) = D(\rho)$ whenever $\|\rho\|=1$.  If in addition $D_2$ is  assumed to be convex, then  recalling that $D_2$ is l.s.c., we conclude that $D_2 = D_2^{**} = C_\infty$.
\end{proof}

\begin{corollary}\label{convexification}\ For every $\rho\in \Prob_-(\R^d)$, we have:
\begin{equation}\label{cinfty-versus-D2}
C_\infty (\rho) \ =\ \min \left\{ \int D_2(Q)\,\nu(dQ)  :  \nu\in \Prob(\Prob_-(\R^d)),\  \int Q \, \nu(dQ)= \rho    \right\} .
\end{equation}
\end{corollary}
\begin{proof} \ The functional  $D_2$ is l.s.c. on $\Prob_-(\R^d)$ which is a metrizable weakly* compact space. 
Therefore, noticing that the barycenter constraint (namely $\int Q \, \nu(dQ)= \rho$) is closed, the infimum in the right hand side of \eqref{cinfty-versus-D2} is actually a minimum. Moreover the infimum value, as a function of $\rho$, is l.s.c. and  coincides  with the convex lower semicontinuous enveloppe of $D_2$, whence \eqref{cinfty-versus-D2}   since  $C_\infty=D_2^{**}$ by  \eqref{dualD}.
\end{proof}


\begin{remark}\label{expansion}
Owing to \eqref{estiMinfty}, we see that, for a bounded cost function $\ell$ and a given potential $v\in C_0(\R^d)$, $M_N(v)-M_\infty(v)$ behaves like $\frac1{N}$ as $N\to \infty$. The identification of the limit of $N (M_N(v)-M_\infty(v))$ as a function of $v$ is an interesting open issue.
Let us stress the fact that the asymptotic behavior of $M_N(v)$ would be different 
 in the case of a highly confining potential $v\notin C_0(\R^d)$, in particular if  $V(x)\eqdef -v(x)$ blows up to $+\infty$ as $|x|\to\infty.$
Indeed, as noticed in \cite[Remark 5.4]{bouchitte2020relaxed}, a connection can be made with the energy of a system of $N$ points subject to an  external confining potential as studied in \cite{SL, serfaty2018systems, petrache2017next} namely:
\[ \mathcal{H}_N(x_1,x_2,\dots, x_N) = \sum_{1\le i<j \le N} \ell(|x_i-x_j|) + N \sum_{i=1}^N V(x_i). \]
By extending the definition of $M_N(v)$ to  unbounded (negative) $v$, we see that:
\[ -M_N(- V) =\frac1{N^2}\,  \inf \gra{\mathcal{H}_N(x_1,x_2,\dots, x_N) : x_i\in \R^d}. \]

In the case of a Riesz cost function $\ell(t)= t^{-s}$  when $d\ge 3$ and $d-2\le s< d$, the above mentioned works reveals an asymptotic behavior as $N\to \infty$ of the form 
\[ \frac{1}{N^2} \min  \mathcal{H}_N = \min \{\mathcal{H}_\infty(\rho): \rho\in \Prob(\R^d)\} + N^{\frac{s}{d} -1} (C_{d,s}(V) + o(1)), \]
where the second minimum is reached by a compactly supported  $\rho\in\Prob(\R^d)$. Since $\min \mathcal{H}_\infty= \inf\{ C_\infty(\rho) + \int V d\rho \st \rho\in \Prob_-(\R^d)\}= - M_\infty(v)$, we are led to the estimate
  $ |M_N(v)- M_\infty(v)|\sim C N^{\frac{s}{d} -1}$ where $1- \frac{s}{d} <1$, in contrast with \eqref{estiMinfty}.
\end{remark}

%
%

\medskip

\begin{remark}\label{comparison}  In the aforementioned work by C. Cotar et al. \cite{cotar2015infinite}, it was proven that $\lim_{N \to \infty} C_N(\rho) = D(\rho)$  for every $\rho\in \Prob(\R^d)$ in the case of a positive-definite cost function in the sense 
that  $x\in \R^d\to \ell(|x|)$ has a non negative Fourier transform. 
Therefore the equality $C_\infty= D^{**}$ obtained in \eqref{dualD} while $\lim_{N \to \infty} C_N(\rho) = C_\infty(\rho)$  for every $\rho\in \Prob(\R^d)$  can be seen as an extension of their result, valid for every pairwise cost function.
Moreover considering the $\Gamma$-limit instead of the simple limit legitimates  the use of infinite marginals energies defined on sub-probabilities.  Note that a similar  $\Gamma$-convergence  proved in \cite{Serfaty2015} when $\Prob(\R^d)$ is equipped with the tight convergence allows to handle only confining potentials.
\end{remark}
 
\begin{remark}\label{positive}
In case of a bounded cost $\ell$, it turns out that requiring that $\ell$ is of positive type (assumption (H5)) is equivalent to the convexity of the two-homogeneous extension $D_2$ as a functional on $\Prob_-(\R^d)$.  Actually, it is proved in \cite{petrache2017next} that the convexity of $D$ is equivalent to a weaker condition, namely that $\ell$ is {\em  balanced positive} in the sense that, for every $m\in \N^{*}$ and every finite subset $\{x_1, x_2,\dots, x_m\} \subset \R^{md} $, it holds:
\begin{equation}\label{balanced}
  \sum_{i,j=1}^m  \ell (|x_i-x_j|) \, t_i t_j \ge 0 \quad \text{whenever } t_1 + t_2 + \dotsb + t_m=0 .
\end{equation}

Accordingly, under \eqref{balanced}, we infer from the assertion (ii) of \autoref{MIformula} that $C_\infty$ coincides with $D$ on $\Prob(\R^d)$. The equality $C_\infty= D_2$ (hence the convexity of $D_2$) will then follow if we know that $C_\infty$ is two-homogeneous. 
Thereby, in view of \autoref{homogeneity-thm}, we may conclude that the conditions (H5) and \eqref{balanced} are equivalent if the  \autoref{k-conjecture} holds true.
\end{remark}

\section{Minimizers and ionization effect} \label{minimizers}

 In this section we focus on the  minimizers of the problem:
\begin{equation}\label{toy}
\inf_{\rho \in \Prob_-(\R^d)} \gra{C_\infty(\rho) -\lambda \bra{v,\rho}}
\end{equation}
 where  $v\in C_0(\R^d)$ is a given external potential and $\lambda>0$ is a strength parameter. As the infimum $- M_\infty(\lambda v)$ depends only on the positive part of $v$, we will assume without loss of generality that 
$ v\ge 0$  and  $\sup v > 0.$
In this case, we know from the \autoref{homogeneity-thm}(ii) that,
under the standing assumptions (including (H4)), we have $M_\infty(v)>0$.  Therefore the minimum of \eqref{toy} is {\em strictly negative}, hence not reached at $\rho=0$.   

We will denote by $\S_\lambda(v)$ the set of minimizers: it is a non empty convex weakly* compact subset of $\Prob_-(\R^d)$.
Notice that $\S_\lambda(v)= \partial M_\infty(\lambda v)$ where $\partial M_\infty$ denotes the subdifferental of $M_\infty$ as a functional on $C_0(\R^d)$.
As stated in the introduction, the problem \eqref{toy} arises in the limit $N\to\infty$ of 
\begin{equation} \label{toyN}
\min_{\rho \in \Prob_-(\R^d)}  \gra{ \overline{C_N}(\rho) - \lambda \bra{v,\rho}}.
\end{equation}

Indeed, by the $\Gamma$-convergence of $\overline{C_N}$ (\autoref{Gconvergence}), the set of minimizers for the relaxed N-marginal problem \eqref{toyN} (namely $\partial M_N(\lambda v)$) converges to $\S_\la( v)$ in the sense of Kuratowski.

Our aim is to characterize situations where one of the following happens:
\begin{enumerate}[(a)]
\item $\S_\la(v) \subset \Prob(\R^d)$ (confinement regime);
\item $\S_\la(v) \subset  \{\|\rho \| < 1\}$  (full ionization);
\item none of (i) and (ii)  (partial ionization).
\end{enumerate}

We expect the existence of nonnegative threshold constants $\la_*(v) \le \la^*(v)$ such that $(a)$ occurs for $\la \ge \la^*(v)$ and $(b)$ for $\la \leq \la_*(v)$. This scenario will be confirmed (see subsection \ref{general-threshold}) with a strict  inequality $\la_*(v)>0$ for rapidly decreasing potenials $v$ (see Subsection \ref{estimates}). Obviously, in the strictly convex case ($\ell$ of strictly positive type), only one of  alternatives (a) and (b) are possible since  $\S_\la(v)$ is a singleton and then $\la_*(v)= \la^*(v)$.
Explicit computations of this common value are provided in the case of a radial potential $v$, $d=3$ and $\ell(r)= \frac1{r}$ (see Subsection \ref{radial}). 

\begin{remark}\label{variantF} Note that \eqref{toyN} is a particular case of
$ \inf_{\rho \in \Prob(\R^d)} \gra{C_N(\rho) + \mathcal{F}(\rho)}$ where $\mathcal{F}$ is a weakly continuous perturbation. 
It turns out that more general terms can be added to the infinite marginal energy, as for instance 
\begin{itemize}
\item  $\mathcal{F}(\rho)= \frac3{5}\int\rho^{5/3} - \la \bra{v, \rho}$ (Thomas-Fermi model (TF)) 
\item $\mathcal{F}(\rho)= \int \frac{|\nabla \rho|^2}{  \rho }- \la \bra{v, \rho}$ (Thomas-Fermi-von Weizs\"acker model (TFW))
\end{itemize}
 However, since in these cases $\mathcal{F}$ is merely weakly l.s.c., the asymptotic analysis requires  further technical arguments which are not studied in this  paper. On the other hand, the ionization issue for the related minimization problems is, to our knowledge, far to be understood as far as general interaction costs $\ell$ and external potential $v$ are considered. In case of  the TF model in dimension $d=3$ ($\ell(r) =\frac1{r}$ and $v$ a Coulomb potential generated by a finite number of nuclei), we refer to the seminal paper by Lieb-Simon \cite{Lieb-Simon77} where the regimes (a) (ionic case) and (b) (neutral case) are evidenced.

\end{remark}

\begin{remark}\label{Frostman}  A lot of work has been dedicated to the existence issue for probabilities (situation (a)) in the case of positive type cost for which $C_\infty= D_2$ and the solution to \eqref{toy} is unique (see for instance \cite{Frostman, Serfaty2015}). In all these works where the external potential $V$ appears with the opposite sign (\ie $V=-v$),  the tightness of minimizing sequences is obtained  under the condition  that the supremum of $S_2 v - c_2$ is stricly negative on the complement of $K^2$ being $K\subset \R^d$ a large compact subset.
Let us notice that the latter condition implies that $\limsup_{|x|\to+\infty} v <0$ which excludes the possibility of considering external potentials $v$ vanishing at infinity.  In this sense the technical arguments used in the aformentioned works needs to be improved  in order to handle the case $v\in C_0(\R^d)$ and  possibly relaxed solutions $\rho\in \S_\la(v)$ such that $\|\rho\|< 1$.
In Subsection \ref{estimates}, we  propose a criterium of slow decay of $v$ at infinity ensuring the confinement scenario (a) for all $\la>0$ while, in the opposite direction, the ionization scenario (b) is shown for every compactly supported $v$ and $\la$
being small enough.
\end{remark}

 \subsection{Existence of  thresholds }\label{general-threshold}

\medskip 
In this subsection, we are in the general framework of a cost function $\ell$ satisfying $(H1)-(H4)$. 
%
For futher discussions it is convenient to introduce  the function $\I: C_0(\R^d) \to [0,1]$ defined by
$$  \I(v) :=\min\{ \|\rho\| : \rho\in \partial M_\infty(v) \} ,$$
where the minimum is attained since $\partial M_\infty(v)=\S_1(v)$ is convex weakly* compact. 
Then, for every $\la\ge0$, the occurrence of scenario (a) (\ie, $\S_\la(v) \subset \Prob(\R^d)$) 
is equivalent to $\I(\la v) = 1$. Accordingly, we define the {\em upper ionization} threshold 
 \begin{equation} \label{def:la^*}
\la^*(v)  \eqdef \sup \gra{ \la >0 \st \I(\la v) < 1}
\end{equation}
and  the {\em lower ionization threshold} under which scenario (b) occurs:
\begin{equation}\label{def:la_*}
\la_*(v) \eqdef \inf \gra{\la > 0 \st S_\la(v) \cap \Prob(\R^d) \neq \emptyset}.
\end{equation}
Clearly one has $\la_*(v) \le \la^*(v)$. In case of a strict inequality,  scenario (c) (partial ionization) will occur for $\la$ in between.
Next, recalling  the monotonicity property of \autoref{subquadra}, we associate to each $v\in C_0(\R^d)$ the  constant:
$$  \K(v):= \sup_{t>0} \frac{M_\infty(t v)}{t^2} \ =\ \lim_{t\to 0_+}  \frac{M_\infty(t v)}{t^2} .$$
and the additional threshold:
\begin{equation}\label{def:ka}
  \kappa(v)  \ :=\  \sup \left\{ \la \ge 0\ :\ \frac{M_\infty(\la v)}{\la^2} = \K(v) \right\},
\end{equation}
where, by  convention, we set $\kappa(v) =0$ if $\K(v)=+\infty$. 
Note that, if $\K(v)<+\infty$, the subset of reals $\la$ appearing above coincides with the interval $[0,\kappa(v)]$.

\begin{thm}\label{2thresholds}\ Let  $v\in C_0(\R^d)$  such that $v\ge 0$ and $\sup v>0$. Then it holds:
$$  0 \le \ \la_*(v) \ \le\ \kappa(v)\le \ \la^*(v) \ <  \ +\infty.$$
Furthermore, if  $\kappa:=\kappa(v)>0$, then  $\S_{\kappa}(v)$ is a subset of $\Prob(\R^d)$ while: 
$$  \I(\la v)  = \frac{\la}{\kappa}\quad \text{and}\quad \frac{\la}{\kappa} \S_{\kappa}(v)\subset  \S_\la(v)  \quad \text{for every} \quad \la\in [0,\kappa].$$
\end{thm}

\begin{remark} \label{linear} In case of a stricly convex functional $C_\infty$  (for instance when $\ell$ is of positive type), the situation gets much simpler since the subset $\S_\la(v)$ reduces to a single element $\rho_\la$ and  the inequalities 
 \eqref{2thresholds} become equalities, \ie:
 $\la_*(v)= \kappa(v)= \la^*(v)$  for every $v\in C_0(\R^d)$. In addition, if $\K(v)<+\infty$,  we 
recover the  linear behavior of $\rho_\la$ observed in  the examples of Subsection \ref{radial}, namely:
 $$  \rho_\la = \frac{\la}{\kappa(v)} \rho_{\kappa(v)}\quad ,\quad  M_\infty(\la v) = \K(v)\, \la^2 \quad \forall \la\in [0,\kappa(v)]\ , $$
 while  $ \|\rho_\la\|= 1 \quad  \forall \la\ge \kappa(v)$. 
 Unfortunately, the two authors did not succeed in  providing any example relying on the non uniqueness of solutions for somes values of $\la$ where the strict inequality  $\la_*(v) <\la^*(v)$ occurs.   
\end{remark}


In order to prove \autoref{2thresholds}, some preliminary steps are in order.
First we relate the set $\S_\la(v)$  to  the solutions of another simpler problem sharing the same infimum,namely
$$ \inf_{\rho \in \Prob_-(\R^d)} \gra{D_2(\rho) - \lambda \bra{v,\rho}}$$
(by \eqref{dualD}, the infimum above \ie $- D_2^*(\la v)$ coincides with $-M_\infty(\la v)$).
Since  $D_2$ is weakly* l.s.c. on $ \Prob_-(\R^d)$, the set of solutions to the latter problem is therefore a non empty compact subset $\widetilde{\S_\la}(v)\subset \S_\la(v)$. The counterpart of the function $\I(v)$ defined above is then
$$ \tilde{\I}(v):=\min\{ \|\rho\| : \rho\in \widetilde{\S_1}(v) \} .$$
Notice that the equality $C_\infty(\rho)=D_2(\rho)$ holds for any element $\rho\in \widetilde{\S_\la}(v)$.

\begin{lemma}\label{nonconvex} \  Let be given $\la \ge 0$ and $v\in C_0(\R^d)$. 
Then  
$$ \S_\la(v) = \ov{\co} \left(\widetilde{\S_\la}(v)\right), \quad \tilde{\I}(\la v) =  \I(\la v) .$$
 \end{lemma}

\begin{proof}\ To establish the first equality, we pick up any element $\rho \in \S_\la(v)$ to which we associate, by means of  the representation formula \eqref{cinfty-versus-D2},  a suitable  $\nu\in \Prob(\Prob_-(\R^d))$ such that  $\rho=\int Q \, \nu(dQ)$. Then we have:  
$$ -M_\infty(\la v) = C_\infty(\rho) - \la \bra{v,\rho} = \int \left( D_2(Q) -   \la \bra{v,Q}   \right)\, \nu(dQ).$$
Since $D_2(Q) -   \la \bra{v,Q} \ge -M_\infty(\la v)$ for every $Q\in \Prob_-(\R^d)$,  the equality above implies that
$Q\in \widetilde{\S_\la}(v)$ for $\nu$-almost all $Q$.  The inclusion 
$ \S_\la(v) \subset \ov{\co} \left(\widetilde{\S_\la}(v)\right)$ follows. Since $\S_\la(v)$ is convex and compact, the converse inclusion is obvious  as well as the inequality $ \I(\la v) \le \tilde{\I}(\la v) $. It remains to show that  $ \I(\la v) \ge \tilde{\I}(\la v) $.
To that aim we select an element  $\rho \in S_\la(v)$ of minimal mass , \ie\  such that that  $\|\rho\|=\I(\la v)$. 
As before, it holds $Q\in \widetilde{\S_\la}(v)$ thus $\|Q\|\ge \tilde{\I}(\la v)$ for $\nu$-almost all $Q\in \Prob_-(\R^d)$.
By integrating with respect to $\nu$, we arrive to  $\I(\la v)=\|\rho\|= \int \|Q\| \, \nu(dQ) \ge  \tilde{\I}(\la v)$. 
This concludes the proof.
\end{proof}

 Next we give the following constancy result which will be crucial:
\begin{lemma}\label{constancy} Let $\la_0>0$ and  $v\in C_0(\R^d)$ be such that $0<\I(\la_0 v)<1$. Then, whenever $\la_0 \le \la < \frac{\la_0}{\I(\la_0 v)}$, we have
  $$ \I(\la v)<1 \ \text{and}\quad \frac{M_\infty(\la v)}{\la^2}= \frac{M_\infty(\la_0 v)}{\la_0^2}  \quad
.$$
\end{lemma}
\begin{proof} \ Let $\rho_0\in \widetilde{S_{\la_0}} (v)$ be such that $\|\rho_0\|= \tilde{\I}(\la_0 v)$ and  let $\la=\frac{\la_0}{\theta}$ where  $\theta$ is any real number such that  $\| \rho_0\| < \theta \le 1$.  Then we have
\begin{align*} M_\infty (\la_0 v)  &= \la_0 \bra{v,\rho_0} - D_2(\rho_0) \\
&= \theta^2  \pa{\bra{\frac{\la_0 v}{\theta},\frac{\rho_0}{\theta}} - D_2(\frac{\rho_0}{\theta})}\\
& \le\  \theta^2\,  D_2^* \pa{ \frac{\la_0 v}{\theta}}\ ,
\end{align*} 
where in the second line, we exploit the two-homogeneity property of $D_2$. 
By the monotonicity \autoref{subquadra} and thanks to the equality $D_2^*=M_\infty$, we infer that the inequality of the last line is actually an equality, whence
$M_\infty(\la v) = \frac {\la^2}{\la_0^2} M_\infty(\la_0 v)$ as claimed . In addition, by setting $\rho:= \frac{\rho_0}{\theta}$  and after dividing by $\theta^2$,
we are led to the relation $\bra{\la v, \rho} - D_2(\rho) = M_\infty(\la v)$. Therefore $\rho\in \widetilde{S_\la}(v)$ and we get:
$$\I(\la v)\ =\ \tilde{\I}(\la v)\ \le\ \|\rho\|\  = \ \frac{\|\rho_0\|}{\theta}\ <\ 1\ .$$
\end{proof}
\begin{proof}[Proof of \autoref{2thresholds}]\ In a first step, we show that $\la^*(v)< +\infty$. Let $\la>0$ such that $\I(\la v)<1$. Then, as $ \widetilde{\I}(\la v)=\I(\la v)$ by \autoref{nonconvex},
there exist $\rho\in \widetilde{S_\la(v)}$ such that $0<\|\rho\|<1$. Then the polynomial  function 
 $$g_\lambda(t) :=  D_2(t\rho) - \lambda \,  \bra{tv,\rho}\ =\  t^2\,  D_2(\rho) - \lambda\, t \,  \bra{v,\rho} $$
 reaches its  minimum on the interval $[0, \| \rho\|^{-1}]$ at $t=1$.
 The necessary condition $g_\lambda '(1)=0$ implies then the equi-partition of energies:
$$ D_2(\rho) = \frac{ \lambda}{2} \bra{v, \rho} = M_\infty(\lambda v), $$ 
from which we deduce
$$ \gra{\la >0 \st \I(\la v) <1 } \subset \gra{ \la >0 \st  M_\infty(\lambda v) \le \frac{ \lambda}{2} \sup v}. $$

Clearly the subset of real numbers in the right hand side above is a finite interval since, by \autoref{CI-dense-domain}, the function $t\mapsto \frac{ M_\infty(tv)}{t}$ is monotone and converges increasingly to $\sup v_+$ as $t\to\infty$.
The desired inequality  $\la^*(v)< +\infty$ follows. 

\medskip
In a second step, we show that $ \la_*(v) \le \kappa(v)$. It is not restrictive to assume that $\la_* \eqdef \la_*(v) > 0$.
Then, obviously,  we have $\I(\la v)<1$ for every $\la<\la_*$. Since $M_\infty$ is convex and Lipschitz, the function $k(t) \eqdef \frac{M_\infty(t v)}{t^2}$ is
locally Lipschitz on $(0, +\infty)$ where it admits left and right derivatives. 
By applying the constancy \autoref{constancy}, we infer that $k'(\la+0) = 0$ for every $\la\in (0, \la_*)$ which in turn implies that $k$ is constant on the whole interval. This constant is finite and coincides with $\K(v) = \lim_{t\to 0_+} k(t)$. In virtue of \eqref{def:ka}, we deduce that $\la_*(v)\le \kappa(v)$.

\medskip
In a third step, we show that $\kappa(v) \le \la^*(v)$. It is not restrictive to assume that $\kappa \eqdef \kappa(v) > 0$.
Let $\rho_*\in S_{\kappa}(v)$. Then we claim that $\rho_* \in \Prob(\R^d)$. Indeed, suppose that $\|\rho_*\|<1$.  Then 
$\I(\kappa v) <1$ and by \autoref{constancy}, we would get that 
\[ \K(v) = \frac{M_\infty(\kappa v)}{\kappa^2} = \frac{M_\infty(\la v)}{\la^2}, \]
for some $\la > \kappa$. This is not possible in view of \eqref{def:ka}. 
Therefore  we have $S_{\kappa}(v) \subset \Prob(\R^d)$ and $\I(\kappa v) = 1$.  The desired inequality 
$\kappa(v) \le  \la^*(v)$ is now straightforward from the definition \eqref{def:la^*}.

\medskip
In a last step, we consider a real $\la$ such that $0\le \la \le \kappa$ and show first that 
$\frac{\la}{\kappa} S_{\kappa}(v)\subset  S_\la(v)$ and then that   $\I(\la v)  = \frac{\la}{\kappa}$.
Letting $\rho_*\in S_{\kappa}(v)$ and $t\in [0,1]$, by using the constancy of the ratio $\frac{M_\infty(\la v)}{\la^2}$ for $\la\in (0, \kappa]$ (see \autoref{constancy}), we have 
\[ D_2(t\rho_* ) - \bra{ \la v,t\rho_*} = t^2 \left( D_2(\rho_*) - \bra{\kappa v,\rho_*} \right)
= - t^2 M_\infty( \kappa v) = - M_\infty( t \kappa v). \]
In particular, for  $t= \frac{\la}{\kappa}$, we infer  that $\frac{\la}{\kappa} \rho_* \in S_\la(v)$ and, thanks to  $\|\rho_*\| = 1$ (see third step), that $\I(\la v)\le \frac{\la}{\kappa}$.
Eventually we claim that a strict inequality $\I(\la v) < \frac{\la}{\kappa}$ is not possible for $\la\in (0, \kappa)$.
Indeed, it would imply the existence of $\rho\in S_\la(v)$ such that $\|\rho \|< \frac{\la}{\kappa}$. Then, by setting  $\mu \eqdef \frac{\rho}{\|\rho\|}$ and $\tilde{\la} \eqdef \frac{\la}{\|\rho\|}$, we would get
\[ M_\infty(\la v) = \la \bra{v,\rho} -C_\infty(\rho) = \|\rho\|^2\, ( \tilde{\la} \bra{v,\mu} - C_\infty(\mu))
\le  \|\rho\|^2\,  M_\infty( \tilde{\la} v), \] 
whence $\K(v)= \frac{M_\infty(\la v)}{\la^2} \le \frac{M_\infty(\tilde{\la} v)}{\tilde{\la}^2}$. 
In virtue of \eqref{def:ka}, the latter inequality  is not  possible since  $\tilde{\la} = \frac{\la}{\|\rho\|}> \kappa$.
The proof of \autoref{2thresholds} is now complete.
\end{proof}

\subsection{Optimality conditions}\label{opticond} \ To every $\rho \in \Prob_-(\R^d)$, we associate
the convolution potential   $u_\rho: \R^d \mapsto [0, +\infty]$ defined by
\begin{equation}\label{def:urho}
u_\rho(x) \eqdef \int \ell(|y-x|) \, \rho(dy) .
\end{equation}
It is a lower semicontinuous function and, as  $D_2(\rho) = \int u_\rho\, d\rho$,  it holds $u_\rho \in L^1_\rho$ whenever $D_2(\rho)$ is finite.

\begin{prop}\label{cnsD} \  Let be given $\la \ge 0$ and $v\in C_0(\R^d)$. 
\begin{enumerate}[(i)]
\item ({\it necessary condition})\ Let $\rho\in \widetilde{\S_\la}(v)$. Then there exists a constant $c_\la$ such that 
\begin{align} \label{cnD} \begin{cases} 
(a)& c_\la \le 0 , \quad  c_\la (1 -\|\rho\|) \ge 0, \\
(b)& u_\rho - \frac{\la}{2} v = c_\la  \quad \text{$\rho$- a.e.},  \\ 
(c)& u_\rho - \frac{\la}{2} v \ge c_\la  \quad \text{$\nu$- a.e} ,\ \text{whenever $D_2(\nu) + \int u_\rho \, d\nu<\infty$.} \end{cases}
 \end{align}

\item Assume that $D_2$ is convex (thus $C_\infty=D_2$ and $\widetilde{\S_\la}(v)=\S_\la(v)$). Then:
  $$\rho\in \S_\la(v)  \ \iff\ \exists c_\la\ \text{such that 
\eqref{cnD}  holds}.$$

\end{enumerate}
 \end{prop}
 
 \begin{remark} \label{equirepar}  By integrating the equality \eqref{cnD}(b) with respect to $\rho$, we obtain the identities \begin{equation}\label{integralident} 
 C_\infty( \rho_\la) = \frac{\la}{2} \bra{v,\rho_\la} + c_\la\, \|\rho\|, \quad M_\infty(\la v)=  \frac{\la}{2} \bra{v,\rho_\la}
 -c_\la\, \|\rho_\la\|, 
 \end{equation}
 for any pair satisfying condition  \eqref{cnD}. 
 \end{remark}  
 
 \begin{proof} First we observe that, for every $\rho\in \Prob_-(\R^d)$,  $u_\rho$ is a lower semicontinuous function from $\R^d$ to $[0, +\infty]$ such that $D_2(\rho) = \int u_\rho\, d\rho$. Hence $u_\rho \in L^1_\rho$ whenever $D_2(\rho)$ is finite. It is the case in particular when $\rho$ is an element of $\widetilde{\S_\la}(v)$ and then, for every non negative measure $\nu$ admissible in the sense of condition (c), we have
  $$ \lim_{\e\to 0_+} \frac1{\e} (D_2( \rho + \e(\nu - \rho)) - D_2(\rho)) =2\, \bra{ u_\rho, \nu}. $$
The  minimality of $\rho$ implies then the following variational inequality:
 $$ \bra{ u_\rho -  \frac{\la}{2}v,  \nu} \ge \bra{u_\rho -  \frac{\la}{2}v, \rho} =\vcentcolon c_\la. $$ 
This  is the exact counterpart of the condition given  in \cite{Frostman, Serfaty2015} which we extend to 
  admissible competitors $\nu\in \Prob_-(\R^d)$. By taking $\nu=0$, we deduce that the constant $c_\la$ defined above
  satisfies  $c_\la\le 0$ while, if $\|\rho\|<1$, we get $c_\la\ge 0$ by taking $\nu= \frac{\rho}{\|\rho \|}$.
  By localizing , we deduce the necessary  optimality conditions \eqref{cnD} given in assertion i).
  Next we observe that, if $D_2$ is convex, then the inequality $ \int u_\rho \, d\nu \le \frac1{2} (D_2(\rho) + D_2(\nu))$ ensures  the validity of condition \eqref{cnD}(c) assuming merely that $D_2(\nu)<+\infty$. 
  Then clearly, conditions \eqref{cnD} imply the variational inequality above which, as well known  for convex problems, characterizes an optimal $\rho$.
  \end{proof} 
   

\begin{corollary}\label{supprho} \ Let $\la>0$ and  $\rho\in \S_\la(v)$. Then $\supp \rho \subset \supp v$. 
\end{corollary}

\begin{proof} \ In view of \autoref{nonconvex}, it is enough to show
  the inclusion $\supp \rho \subset \supp v$  when  $\rho\in \widetilde{\S_\la}(v)$.
 By appying the optimality condition in \eqref{cnD}(b) of \autoref{cnsD} and since $c_\la\le 0$, we infer that $u_\rho \le \frac{\la}{2} \, v$ holds $\rho-$a.e.
On the other hand, for any   $r>0$ ,  we have the inequality: 
 $$ u_\rho(x) \ge \int_{B(x,r)} \ell(|y-x|) \, \rho(dy) \ge \underline{\ell}(r) \rho (B(x,r)), $$
 being  $\underline{\ell}(r)$ defined by \eqref{def:linf-lsup}.   
It follows that $A_r \eqdef \left\{x\in\R^d : \ \rho(B(x,r))  >   \frac{\la}{2} \ \frac{v(x)}{\underline{\ell}(r)}\right\}$
is $\rho$-negligible. As can be readily checked, the map $x\mapsto  \rho(B(x,r))$ is lower semicontinuous. Hence
by the continuity of $v$, $A_r$ is an open subset of $\R^d\setminus\supp(\rho)$ and we are led to the inequality 
\begin{equation} \label{estiball}
\rho(B(x,r)) \leq \frac{\la}{2} \frac{v(x)}{\underline{\ell}(r)}, \quad \forall x\in \supp(\rho) .
\end{equation}
By \autoref{ellem}, we know that  $\underline{\ell}(r)>0$ for small values of $r$. Therefore
 the inclusion $\supp \rho \subset \supp v$  is straightforward by applying \eqref{estiball}.  
\end{proof}

\subsection{Estimates for the threshold values}\label{estimates}

Let us  associate with the cost $\ell$ 
its lower and upper non increasing counterparts:
\begin{equation}\label{def:linf-lsup}
\underline{\ell}(r) \eqdef  \inf \{ l(t) : t\le  r\}, \quad  \ov{\ell}(r) \eqdef  \sup \{ l(t) : t\ge  r\}.
\end{equation}
 In the forthcoming estimates, we will also need the average of a radial function $f(|x|)$ on a ball, namely: 
 \begin{equation}\label{def:avl}
 [f](r) \eqdef \frac{d}{r^d} \int_0^r f(t)\, t^{d-1} \, dt = d \int_0^1 f(r s)\,  s^{d-1}\, ds.   \end{equation}
By \eqref{def:Krd}, the function $[\ell](r)$ is finite on $(0,+\infty)$ and  it is easy to check that:
$$ \underline{\ell}(2 r) \le\ K(r,d)  \le 2^d\, [\ell](2r). $$
 \begin{lemma}\label{ellem}  The functions $ \underline{\ell},  \ov{\ell}, [\underline{\ell}],  [\ov{\ell}]$ are l.s.c. monotone non increasing
and vanish at infinity. Moreover $\underline{\ell}$ is positive at least on a small interval $[0,\delta]$. 
\end{lemma}
\begin{proof} \ By (H2), the infimum in \eqref{def:linf-lsup} is a minimum, thus $\underline{\ell}$ is strictly positive on a small intervall $[0,\delta]$ since otherwise
we would have $\ell(0)=0$. On the other hand, for any given $\a>0$, the inequality $\ov{\ell}(r)\le \a$ holds if and only if the open subset $\{\ell>\a\}$ is contained in $[0,r)$. This property is clearly stable under convergence of a sequence $r_n\to r$, whence the lower semicontinuity property of $\ov{\ell}$.  
The monotonicity property of $[\underline{\ell}],  [\ov{\ell}]$ can be deduced from the second equality in \eqref{def:avl} when choosing  $f$ to be  non decreasing. On the other hand, passing to the limit $R\to\infty$ in the inequality $ [\ov{\ell}](R) \le  \frac{d}{R^d}  ( \int_0^{R_0}   \ov{\ell}(r) \, r^{d-1}\, dr) +  \ov{\ell}(R_0)$, and then as $R_0\to +\infty$,
we deduce that $[\ov{\ell}](R)\to 0$ as $R\to\infty$. Eventually we get the lower semicontinuity of $[\ell], [\underline{\ell}],[\ov{\ell}]$  by applying  Fatou's lemma.
\end{proof}

We are now in position to give a lower estimate for the lower threshold $\la_*(v)$ defined in \eqref{def:la_*}.
For every $v\in C_0(\R^d ; \R^+)$, we set $\alpha_v\in  (0, +\infty]$ to be defined by  
\begin{equation}\label{def:alpha_v}
\alpha_v \eqdef \sup \left\{ \frac{v(x)}{\underline{\ell}(2 |x|)}, x\in \R^d \right\}6.
\end{equation}

\begin{prop} \label{esti-la_*}  \ Let  $\alpha_v$ given in \eqref{def:alpha_v}.
 Then we have  the inequality: 
 $$\la_*(v) \ge \frac{2}{\alpha_v}.$$ 
 \end{prop}
  In view of the lower bound given above , the condition  $\alpha_v<+\infty$ is sufficient to predict  
 full  ionization (scenario (b)) for small values of $\la$. This applies in particular when $v$ is compactly supported 
 or more generally when $\limsup_{|x|\to \infty}  \frac{v(x)}{\underline{\ell}(2 |x|)}<+\infty.$
 
 Note however that the estimate of \autoref{esti-la_*} is not optimal in view of the examples given in the radial case (see Subsection \ref{radial} where $\ell(r) = \frac1{r}, v(x)=V(|x|)$). In this examples, an exact threshold value \eqref{laV=} can be derived in term of a constant $\a_V$ similar to $\a_v$  (see \eqref{alpha-r*}) but not including the doubling factor as in \eqref{def:alpha_v}.

\begin{proof} \ Assume that $\a_v<+\infty$ and let $\la \in (0,\frac{2}{\alpha_v})$. Using \autoref{nonconvex}, it is enough to show
 that any $\rho\in \widetilde{\S_\la}(v)$ satisfies $\|\rho\|<1$. 
Let us apply the inequality \eqref{estiball} at any point  $x\in \supp(\rho)$
with  $r=2 |x|$.  In view of \eqref{def:alpha_v}, we deduce that:
$$ \rho(B(0, |x|)) \le \rho(B(x, 2|x|)) \le \frac{\la}{2} \ \frac{v(x)}{\underline{\ell}(2 |x|)} \le \frac{\la}{2} \, \a_v \quad \forall x\in \supp(\rho). $$
By  selecting a sequence $x_n \in  \supp(\rho)$ such that $\rho(B(0,x_n)) \to \|\rho\|$, we  may conclude that
$\|\rho\| \le \frac{\la}{2} \, \a_v <1$. 
\end{proof}

In the last part of this subsection, our aim is  to provide an upperbound for the  upper threshold $\la^*(v)$ defined in \eqref{def:la^*}. By \autoref{2thresholds}, we know already that $\la^*(v)<+\infty$. An interesting issue would be in particular to characterize potential $v$ for which $\la^*(v)=0$, meaning that $\S_\la(v)\subset \Prob(\R^d)$ for every $\la >0$ (we will say that $v$ is {\it strongly confining}).
From now on, we require an additional assumption ensuring a slow decay of $\ell$ at infinity,
 namely we assume that it exists a Borel function $k_\ell: (0,1]\to \R_+$ such that:
 \begin{equation}\label{Kell}
\begin{cases} \ell(sr) \le k_\ell(s) \ell(r), & \forall s\in (0,1], \forall r>0; \\ 
 k_\ell(1^-)=1; & I(k_\ell) \eqdef \int_0^1 k_\ell(s) s^{d-1}\, ds < +\infty.
 \end{cases}
\end{equation}
This assumption is  satisfied  for $\ell(r)= r^{-p}$ with $0<p<d$. However it rules out exponentially decreasing 
costs (like $\ell(r) = \frac{\exp(-\eps r)}{r}$), for which we expect that $\la^*(v)=0$ is not possible for $v\in C_0$ unless $v\equiv 0$. 
  
 To every $v\in C_0(\R^d ; \R^+)$, we associate the constant $\beta_v\in  (0, +\infty]$  defined by
 \begin{equation}\label{def:beta_v}\ 
\beta_v \eqdef \liminf_{|x|\to \infty} \  \frac{v(x)}{\ov{\ell}(|x|)}. \end{equation}
Note that $\beta_v \le \alpha_v$ as a consequence of the inequality  $\ov{\ell}(r)\ge \underline{\ell}(2 r)$. 
\begin{prop} \label{esti-la^*} Let $\ell$ satisfy \eqref{Kell} and  $\beta_v$ be given by \eqref{def:beta_v}. Then
 $$ 0 \le \la^*(v) \le \frac{2}{\beta_v}. $$ 

 Accordingly $v$ is strongly confining if $\beta_v=+\infty$
 \end{prop}
\begin{proof} By \autoref{nonconvex} and the definition \eqref{def:la^*}, we are reduced to prove the following implication:
$$\widetilde{\I}(\la v)<1 \ \implies  \la \le \frac2{\beta_v}.$$
Let $\la >0$ and  $\rho\in \widetilde{\S_\la}(v)$ with $\|\rho\|<1$. 
We fix a positive  $\eta>1$ and denote by $\theta_{R,\eta}$ the uniform probability  on the crown
$S_{R,\eta}:=\{R<|x|< \eta R\}$. By applying the optimality condition (c) and integrating with respect to $\nu =\theta_{R,\eta}$, we get:
\begin{equation}\label{cn=}
\bra{ u_\rho, \theta_{R,\eta}} = \int w_{R,\eta}\, d\rho  \ = \frac{\lambda}{2} \bra{ v, \theta_R},
\end{equation}
   being  $w_{R,\eta}(x) \eqdef \int \ell(|x-y|) \, \theta_{R,\eta}(dy) $ the radial potential associated with $\theta_{R,\eta}$. 
 We claim that, for every $\eta>1$,
 \begin{equation}\label{crucial}
\liminf_{R\to \infty} \frac{ \bra{v, \theta_{R,\eta}}}{\ov{\ell} (R)} \ge \beta_v k_\ell(\eta^{-1}), \quad \limsup_{R\to \infty} \frac{\int w_{R,\eta}\, d\rho}{\ov{\ell} (R)} \le \norm{\rho}.
\end{equation}
Under  \eqref{crucial},  by passing to the limit as $R\to \infty$ in the equality \eqref{cn=} after dividing by $\ov{\ell} (R)$, we get:
$$  \beta_v \ k_\ell(\eta^{-1}) \le \frac{2}{\lambda} \norm{\rho} < \frac{2}{\lambda}. $$
The desired inequality follows ultimately as $\eta\to 1^+$ since $k(1^-)=1$. 

\med
{\em Proof of claim \eqref{crucial}}\ We prove first the left hand side inequality. It is not restrictive to assume 
that $\beta_v>0$. Let $\beta\in (0,\beta_v)$. Then, recalling that $\ov{\ell}$ is non increasing,  we have
for large values of $R$:
$$ \bra{v, \theta_{R,\eta}} \ge  \beta \bra{\ov{\ell}(|x|), \theta_{R,\eta}} \ge \beta \ov{\ell}(\eta R) \ge \beta k_\ell(\eta^{-1}) \ov{\ell}(R), $$
where, in the last inequality, we used the fact $\ov{\ell}$ satisfies the same inequality as $\ell$ in \eqref{Kell} while keeping the same factor $k_\ell(s)$ (here $s=\eta^{-1}$).
The desired inequality follows as $R\to \infty$ after dividing by $\ov{\ell}(R)$ and by letting $\beta\nearrow \beta_v$.  

The proof of the second inequality in \eqref{crucial} is more technical. Let us fix $R_0>0$ such that $R_0<R$.
Then, since $\ell([y-x|) \le \ov{\ell}(R-R_0)$ whenever $|x|\le R_0$ and $y\in \supp(\theta_{R,\eta})$,  we have
\begin{align*}
 \frac{ \bra{w_{R,\eta}, \theta_{R,\eta}}}{\ov{\ell} (R)} & \le \frac{\ov{\ell}(R-R_0)}{\ov{\ell}(R)} \rho(|x|\le R_0) + \frac{\sup w_{R,\eta}}{\ov{\ell}(R)} \rho(|x|> R_0)\\
 & \le k\pa{1-\frac{R_0}{R}} \rho(|x|\le R_0) + C_\eta\, \rho(|x|> R_0) ,
\end{align*}
where we have set $\displaystyle C_\eta \eqdef \sup_R \frac{\sup\, w_{R,\eta}}{\ov{\ell}(R)}$.
If we can show that $C_\eta<+\infty$, then the desired inequality will clearly follow by sending $R\to\infty$ 
and then $R_0\to\infty$. 
Clearly $w_{R,\eta}\le \ov{w}_{R,\eta}$ where $\ov{w}_{R,\eta}(x) \eqdef \int \ov{\ell}(|x-y|) \, \theta_{R,\eta}(dy) $.
Since $\ov{\ell}$ is non increasing, the global supremum of the radial function $\ov{w}_{R,\eta}(x)$ is reached on $S_{R,\eta}$. For $x, y$ in $S_{R,\eta}$, it holds $|x-y|\le 2 \eta R$ so that for every $x\in S_{R,\eta}$, we have
$$\ov{w}_{R,\eta}(x) \le \frac1{(\eta^d-1) \omega_d R^d} \int_{|z|\le 2\eta R} \ov{\ell}(|z|)\, dz \le \frac{2^d}{(\eta^d-1) }  [\ov{\ell}](2\eta R). $$
The last condition in \eqref{Kell} implies that $[\ov{\ell}](2\eta R) \le [\ov{\ell}]( R) \le I(k_\ell)\, \ov{\ell}( R)$, hence we have $C_\eta \le \frac{2^d}{(\eta^d-1) }\, I(k_\ell) <+\infty$.
 \end{proof}

 \subsection{Radial examples}\label{radial}
 
In this subsection we assume that
\begin{equation}\label{data}
 d=3,\quad \ell(r)= \frac1{r},\quad v(x)= V(|x|),
\end{equation}
where $V: \R_+ \to [0,+\infty)$ is continuous locally Lipschitz such that 
$$\lim_{r\to\infty} V(r)= 0.$$

\medskip
As  the Coulomb cost $\ell(r) =\frac1{r}$ is of positive type on $\R^3$,  the extended direct energy $D_2$  is a stricly convex functional on $\Prob_-(\R^3)$. Therefore $C_\infty= D_2$ (see \autoref{MIformula}(ii)) and, for every $\la\ge 0$, we have $\S_\la(v)=\{ \rho_\la\}$ where $\rho_\la$ is the unique minimizer
of \eqref{toy}. Accordingly, the two thresholds introduced in \eqref{def:la_*}\eqref{def:la^*} are 
equal and coincide with
$$  \la_V := \min\{ \la \ge 0:  \|\rho_\la\|=1 \}.$$

By using the invariance under rotations and te convexity of $C_\infty$, it is easy to show that $\rho_\la$ is  radial. Therefore $\rho_\la$ is determined by the cumulative distribution function:
 $$F_\la(r):= \rho_\la(\{|x|\le r\}),$$
which is  non increasing in $[0,+\infty]$ and satisfies
\begin{equation}\label{def:rla}
F_\la(+\infty) = F(r_\la) = \norm{\rho_\la} \quad \text{where} \quad r_\la \eqdef \sup\{|x| : x\in \supp(\rho_\la)\}.
\end{equation}
 
Next we introduce the two parameters $\a_V$ and $r_V^*$ defined by:
\begin{equation}\label{alpha-r*}
\alpha_V \eqdef \sup_{r>0} r V(r), \quad r_V^* \eqdef \inf \{ r>0  : r V(r) = \alpha_V \}, 
\end{equation}
(with the convention that $r_V^*=+\infty$ if the infimum is not attained).

\medskip
We are going to show that the threshold value $\la_V$ is given by
\begin{equation} \label{laV=} \la_V  = \frac{2}{\alpha_V} 
\end{equation}
and that, for every $\la > \la_V$, the effective radius $r_\la$ given by \eqref{def:rla}) satisfies $r_\la<r_V^*$ 
(hence finite) while $r_\la =r_V^*$ for every $\la\le \la_V$.

\medskip A key argument of the proof relies on the fact that  $\frac1{4\pi r}$ is a fundamental solution of the 3d Laplacian; then $C_\infty(\rho)<+\infty$ implies that $\rho\in W^{-1,2}(\R^3)$ while
$$ C_\infty(\rho)\ =\  D_2(\rho)\ =\ \bra{ u_\rho, \rho}\ =\ \frac1{4\pi}\int |\nabla u_\rho|^2 \, dx, $$
where $u_\rho$ is the Coulomb potential generated by $\rho$, \ie,  the solution of
\begin{equation}\label{Poisson}
-\Delta u_\rho = 4\pi \rho, \quad u_\rho\to 0 \quad \text{as}\ |x| \to \infty.
\end{equation}

Note that here  $u_\rho$ is defined a.e. in the sense of the Newtonian capacity, hence admits a representative defined $\nu$-a.e., where $\nu$ is any element of $\Prob_-(\R^d)\cap W^{-1,2}(\R^d)$.
 Particularizing to $\rho$ being the solution $\rho_\la$, we get a radial potential in the form 
 $ U_\la(|x|)$, where $U_\la: \R_+\to \R_+$  is continuous non increasing and $U_\la(+\infty)=0$. 
  Then the  Laplace equation \eqref{Poisson} holding in the distributional sense in $\R^3$  implies that  $F_\la$ is a primitive of the non negative measure $(-r^2 U_\la')'$. More precisely, $U_\la$ can be deduced from $F_\la$
  by the conditions
 \begin{equation}\label{Ula}
F_\la(r) = - r^2 U_\la'(r-0), \quad F_\la(r+0) = - r^2 U_\la'(r+0) \quad \forall r>0 . 
\end{equation}
which, in the case $r_\la<+\infty$, imply that
$$ U_\la(r) = \frac{\|\rho_\la\|}{r} \quad \forall r\ge r_\la. $$

\medskip 

In order to simplify the computations, we make the additional assumption that the  potential
 $V$  satisfies:
\begin{equation} \label{Lap>0}
  - r^2 V'(r) \text{ is monotone non decreasing on $\{V>0\}$}.
\end{equation}
This implies that $V$ is non increasing, hence  $\{V>0\}$ is an open interval $[0, r_V)$ ($r_V\in (0,+\infty]$)  and $V'$ admits a left $V'(r- 0)$ at any $r\in(0,r_V]$  (resp. a right derivative  $V'(r+ 0)$ at any $r\in[0,r_V)$). 
Moreover the distributional Laplacian  $-\Delta v$ is a non negative measure on the ball $\{|x|<r_V\}$.\footnote{The case where $-\Delta v$ exhibits a negative part is more tricky and not addressed in this paper.}
We define a monotone non increasing function $\zeta_V:\R_+\to [0,+\infty]$ by setting:
\begin{equation}\label{def:zeta}
\zeta_V(\la) \eqdef \sup \left\{ r>0 : -r^2\, V'(r-0)< \frac{2}{\la}\right\}.
\end{equation}

\begin{lemma} \label{ineq.ab} Let $\a_V, r_V^*, \la_V$ be given by \eqref{alpha-r*}\eqref{laV=}. Then, under \eqref{Lap>0},  we have: 
\begin{equation}\label{zeta=}
 \zeta(\la_V) = r_V^*, \quad \zeta(\la) < +\infty \quad \forall \la>\la_V,
\end{equation}
\begin{equation}
\label{rV0}
\lim_{\la\to +\infty} \zeta(\la) = r_V^0 \eqdef \max\{r\ge 0 :  V(r)=V(0)\}.
\end{equation}

 \end{lemma} 
 \begin{proof}  By  \eqref{Lap>0}, for every finite $\la>0$ and $r\in [0,+\infty)$, we have
\begin{align}
\zeta(\la) = r  &\iff \frac2{\la} \in [-r^2 V'(r-0),-r^2 V'(r+0)], \label{testzeta} \\
\zeta(\la)=+\infty  &\iff \frac2{\la} \ge \gamma_V \eqdef  \sup \{-r^2 \, V'(r): r\ge 0\}. \label{zeta=infty} 
\end{align}

If $r_V^*<+\infty$, the optimality of $r_V^*$ in \eqref{alpha-r*} implies that 
 $$r_V^* V'(r_V^*+0)+ V(r_V^*) \le 0 \le r_V^* V'(r_V^*-0)+ V(r_V^*),$$
  hence  after multiplying by $r_V^*$ and since $ r_V^* V(r_V^*)= \a_V= \frac2{\la_V}$:
  \begin{align} \label{ineq1.alpha}
 - (r_V^*)^2 \, V'(r_V^*-0)  \le  \frac2{\la_V} \le  - (r_V^*)^2 \, V'(r_V^*+0). 
 \end{align}
 We thus conclude \eqref{zeta=} in the case $r_V^*<+\infty$.
 
 If $r_V^*=+\infty$, then  $\{V>0\}=\R_+$  and $-r^2 V'$ is monotone increasing up to 
 $\gamma_V$. It follows that $- V'(t) \le \frac{\gamma_V}{t^2}$ for every $t>0$,  while for arbitrary $\gamma<\gamma_V$, we have  $- V'(t) \ge \frac{\gamma}{t^2}$ for large $t$. 
From $V(r) = \int_r^{+\infty} -V'(t) dt\ $, we infer  that: 
$$ r V(r) \le \gamma_V \quad \forall r\ge 0, \qquad \liminf_{r\to +\infty} rV(r) \ge \gamma \quad \forall \gamma<\gamma_V ,$$
whence  the equalities  $\a_V=\gamma_V= \frac2{\la_V}$. 
Then from  \eqref{zeta=infty} we recover \eqref{zeta=}, since the equality $\zeta(\la)=+\infty$ is equivalent to   $\la\le \la_V$. 

 Eventually in order to to show \eqref{rV0}, we notice that for every $\la>0$ we have $\zeta(\la)\ge r_V^0$ since
 $-r^2 V'(r-0)$ vanishes on $[0,r_V^0]$. In the opposite direction, les us choose a sequence $\la_n \searrow +\infty$ such that $r_n:=\zeta(\la_n)\to \limsup_{\la\to+\infty} \zeta(\la)$. Then $r_n$ converges  decreasingly to some $r_*$ such that $r_*\ge r_V^0$.  By \eqref{ineq1.alpha}, we have $-r_n^2 V'(r_n-0)\le \frac2{\la_n}$.
 Passing to the limit $n\to\infty$, it follows that $-r^2 V'(r)$ vanishes at $r_* =0$ , hence on the whole interval
 $[0, r_*]$ by the monotonicity assumption.  Owing to the definition of $r_V^0$ (see \eqref{rV0}), this is possible only if $r_*=r_V^0$. 
  \end{proof}

 Now we may construct a radial ground state probability $\rho_V^*$ supported in $\{|x|\le r_V^*\}$. Its distribution function is given by:
\begin{equation}\label{rho*}
F_V^*(r) \ := \begin{cases}  \frac1{\alpha_V}\, (- r^2 V'(r-0)) &  \text{if $r<r_V^* $}\\
1 & \text{if $r\ge r_V^* $}
\end{cases}
\end{equation} 
which, by \autoref{ineq.ab}, satisfies  $ F_V^*(+\infty)= F_V^*(r_V^*)=\|\rho_V^*\| =1$.
Note that, if $r_V^*<\infty$, a positive jump of $F_V^*$ may occur (but not only) at $r=r_V^*$ corresponding to a concentration of  $\rho_\la$ on $\{|x|=r_V^*|\}$.


It turns out that, for $\la \in [0,\la_V]$, the total mass of $\rho_\la$ grows linearly
from $0$ to $1$ (thus we have a ionization regime for $\la<\la_V$ if $\alpha_V<+\infty$)
while the support of $\rho_\la$ is a ball of constant radius $r_V^*$. 
In contrast, for $\la \in (\la_V,+\infty)$, $\rho_\la$ is a probability supported in  a  ball of radius $r_\la$  which,  as $\la$ increases to $+\infty$, converges decreasingly  to  $r_V^0$ given in \eqref{rV0}.
The same kind of picture can be observed in the Thomas-Fermi theory (see \cite{Lieb-Simon77}).

 \begin{thm}\label{easytest} Under the assumptions \eqref{data} and \eqref{Lap>0},  we have the following alternative in terms of the threshold value $\la_V$ given in \eqref{laV=}: 
\begin{enumerate}[(a)]
\item Let $\la > \la_V$.  Then $\|\rho_\la\|=1$ and it holds:
\begin{equation}\label{Fla=}
 F_\la(r) =  \min \left\{ \frac{\la}{2} (- r^2 V'(r-0)), 1\right\} \quad \forall r>0.
\end{equation}
 The radius of the ball supporting $\rho_\la$  is given by $r_\la= \zeta(\la)$, being $\zeta$ defined by \eqref{def:zeta}. It is finite for every  $\la> \la_V$ and not larger than $r_V^*$ given in \eqref{alpha-r*}.
\item Let $\la\in [0,  \la_V]$. Then, with $\rho_V^*$ defined by \eqref{rho*}, we have
 $$\rho_\la=\frac{\la}{\la_V} \, \rho_V^* .$$
In particular, if $\alpha_V<+\infty$ , then ionization occurs if and only if $\la <\la_V$. 
\end{enumerate}
\end{thm}

\begin{proof} \ We will be using the sufficient part of the optimality conditions obtained in \autoref{cnsD}, which applies in particular to our radial solution $\rho_\la$ (recall that, in our case,  $\widetilde{\S_\la}(v)= \S_\la(v)$). Clearly these conditions are satisfied if we can check that, for a suitable $r_\la \in [0,+\infty]$, we have:
 \begin{equation}\label{optila}
U_\la -  \frac{\la}{2} V  = c_\la  \quad \forall r \le r_\la, \quad U_\la - \frac{\la}{2} V \ge c_\la \quad \forall r\ge r_\la,
\end{equation}
 where $c_\la\le 0$ is a suitable constant such that $c_\la=0$ if $\|\rho_\la\|<1$ or if $r_\la=+\infty$. 

\paragraph{\bf First case: $\la > \la_V$} We may take $r_\la=\zeta(\la)$ which, according to \autoref{ineq.ab},
is finite not larger than $r_V^*$ and  decreases to $r_V^0$ as $\la\to\infty$. Moreover, thanks to \eqref{testzeta},  it satisfies $- \frac{\la}{2}r_\la^2 V'(r_\la-0) \le 1 \le - \frac{\la}{2}r_\la^2 V'(r_\la+0)$. This ensures that $F_\la$ given in \eqref{Fla=} is associated to a probability $\rho_\la$ and we have
$$  F_\la(r) = \frac{\la}{2} (- r^2 V'(r-0))  \quad \text{if $r\le r_\la$,} \qquad F_\la(r) = 1  \quad \text {if $r\le r_\la$.} $$
We have  now to show that such  $\rho_\la$ fulfills the optimality criterium \eqref{optila}.
 From \eqref{Ula}, we know  that the potential $U_\la$ generated by $\rho_\la$ satisfies
 $$  - r^2 U_\la'(r-0) = -\frac{\la}{2} r^2 V'(r-0)  \quad \forall r \le r_\la, \quad U_\la(r) = \frac1{r} \quad \forall r\ge r_\la. $$
 It follows that 
$$ U_\la (r) =\begin{cases} 
  c_\la + \frac{\la}{2} V(r) & \text{if $r<r_\la$} \\
  \frac{1}{r} & \text{if $r>r_\la$,} \end{cases} $$
where  $c_\la = \frac1{r_\la} - \frac{\la}{2} V(r_\la)$ has been chosen so that  $U_\la$ is continuous at $r_\la$. Since $\la>\la_V$ and recalling that, by \eqref{alpha-r*}, we have $r_\la V(r_\la) \le \a_V$,
we deduce that $c_\la< 0$. It remains to check  the inequality part of 
\eqref{optila}. To that aim we consider the function $f(r) \eqdef \frac1{r}-\frac{\la}{2} V(r) $.
whose right derivative exists and is given by $f'(r+0) = - \frac{\la}{2} V'(r+0) - \frac1{r^2}$.
By \eqref{testzeta} and  \eqref{Lap>0},  for every $r\ge r_\la$, we have $- \frac{\la}{2} r^2 V'(r+0)\ge - \frac{\la}{2} r_\la^2 V'(r_\la+0) \ge 1$, hence $f'(r+0) \ge 0$. The wished inequality follows since $f(r_\la) =c_\la$ 
(in fact $f$ reaches its global maximum at $r=r_\la$). 

\paragraph{\bf Second case: $\la \le \la_V$} Since $\rho_\la\equiv 0$ for $\la =0$, it is not restrictive to assume that 
$\la_V>0$ (\ie, $\a_V<+\infty$) while $ 0<\la\le \la_V$. By construction the measure $\rho_\la:=\frac{\la}{\la_V} \, \rho_V^*$ satisfies $\|\rho_\la\|=\frac{\lambda}{\la_V}$ and is supported on the closed ball of radius $r_\la= r_V^*$.
Therefore, from \eqref{Ula} and the definition \eqref{rho*}, its potential $U_\la$ is characterized by the equalities:
 $$  - r^2 \, U_\la'(r-0) = -\frac{\la}{2} r^2 \, V'(r-0)  \quad \forall r\le r_V^*, \quad
  U_\la(r) = \frac {\la }{\la_V} \frac1{r} \quad \forall r\ge r_V^*.$$
If $r_V^*<+\infty$, by taking into account the continuity of $U_\la$ is continuous at $r=r_V^*$, we obtain
$$ U_\la (r) =\begin{cases} 
c_\la + \frac{\la}{2} V(r) & \text{if $r<r_V^*$} \\
  \frac{1}{r} & \text{if $r>r_\la$} \end{cases} $$
with $c_\la= \frac{\la}{\la_V}  \left(\frac1{r_V^*} - \frac{\la_V}{2} V(r_V^*)\right)$.
The minimality of $r_V^*$ in \eqref{alpha-r*}  and \eqref{laV=} implies that $c_\la=0$ while, for any  $r\ge r_V^*$,
it holds $  U_\la(r) - \frac{\la}{2} V(r) =   \frac{\la}{2r}  (\a_V- r V(r)) \ge 0 .$
Therefore the optimality condition \eqref{optila} is satisfied and $\rho_\la$ is optimal. 
If $r_V^*=\infty$, we are led to the equality  $U_\la = c_\la + \frac{\la}{2} V(r)$ holding for every $r>0$. Then, by  sending $r\to \infty$ and since  $U_\la$ and $V$ are vanishing at infinity, we recover that $c_\la=0$
and still obtain that \eqref{optila} is satisfied.
This concludes the proof of \autoref{easytest}.
\end{proof}

We end this subsection with  several  examples where the different situations described in \autoref{easytest} appear explicitely: a positive threshold $\la_V$ determining two regimes appear in the first three examples. In the third one 
the solution is not compactly supported whenever  $0<\la<\la_V$, while  in the last example the potential is  confining for every $\la>0$. 

 Note that, by means of the relations given in \autoref{equirepar}, we may explicit the opposite of the minimum value of \eqref{toy} as follows:
 \begin{equation} \label {Minfty-rad}  M_\infty(\la v) = \begin{cases}
 		\frac{\la}{2} \, \bra{v,\rho_\la}  & \text{if $0\le \la \le \la_V$} \\
 		\frac{\la}{2} \, \bra{v,\rho_\la} +  \frac{\la}{2} \, V(r_\la) - \frac1{r_\la} & \text{if $\la \ge \la_V$,}
\end{cases}
\end{equation}
where the bracket  $\bra{v,\rho_\la}$ can be recovered for every $\la\ge \la_V$  by using the formula
\footnote{We use that $\bra{v,\rho_\la} =\int_0^\infty \rho_\la( \{v\ge r\})\, dr$ where 
$\rho_\la( \{v\ge r\}) = F_\la(t)$ being $t=V^{-1}(r)$.}
\begin{equation} \label {trick} 
 \bra{v,\rho_\la} = -\int_0^{+\infty} F_\la (t) \, V'(t)\, dt  =  -\int_0^{r_\la} F_\la (t) \, V'(t)\, dt 
+  \, V(r_\la). 
\end{equation}
The relations \eqref{Minfty-rad}\eqref{trick} are used for computing $M_\infty(\la v)$ in the four examples below
and we recover the behavior  of $M_\infty(\la v)$ as $\la\to \infty$ and as $\la\to 0$
predicted by \autoref{CI-dense-domain} and \autoref{slopeM} respectively.
Moreover, in the three first examples, we observe the quadratic behavior of $M_\infty(\la v)$ with respect to $\la $ for $\la <\la_V$ as it was predicted in \autoref{linear}.

\begin{example}\label{V1}  \ Let $V(r)= (1-r^2)_+$. Then
$$   \alpha_V= \frac{2}{3\sqrt{3}}, \quad \la_V= 3\sqrt{3}, \quad r_V^*= \frac1{\sqrt{3}}, \quad r_V^0=0.$$
Accordingly, we have $r_\la= \min\{\frac1{\sqrt{3}}, \la^{-\frac1{3}} \}$
and $\rho_\la $ is given by
$$ \rho_\la = \frac{3\la}{4\pi} \, \L^3 \res B(0, r_\la). $$
The minimum value of \eqref{toy} is the opposite of:
$$ M_\infty(\la v) = \begin{cases} \frac{2\, \la^2} {15 \sqrt{3}}   & \text{if $\la \le 3 \sqrt{3}$}\\
\la - \frac9{5} \la^{\frac1{3}}   & \text{if $\la \ge  3 \sqrt{3}$}
 \end{cases} $$

This example shows that the inclusion  $\supp \rho \subset \supp v$ of \autoref{supprho} may be strict, as it is here the case for $\lambda > \sqrt{3}$.
\end{example}

\begin{example}\label{V2}  
Let $V(r) = \frac1{r} \wedge 1$. Then: 
$$   \alpha_V= 1, \quad \la_V=2, \quad r_V^*=r_V^0= 1 .$$ 
Note that here the maximum of $r V$ is reached at any $r\ge 1$. We get $r_\la=1$  for every $\la>0$    and
$$ \rho_\la = \frac{\la \wedge 2}{8\pi}\,  \mathcal{H}^2 \res \{|x|=1\}, $$
while 
$$ M_\infty(\la v) = \begin{cases} \frac{\la^2}{4}  & \text{if $\la \le 2$}\\
\la -1  & \text{if $\la \ge  2$.}
 \end{cases} $$

In this example  $\rho_\la$ concentrates on the interface $|x|=1$ where  $V'$
exhibits a jump. Surprisingly $\rho_\la$ remains constant when $\la$ is beyond the ionization threshold $\la_V=2$.
\end{example}

\begin{example}\label{V3}  
Let $V(r) = \frac1{1+r}$. Then: 
 $$ \alpha_V = 1, \quad \la_V = 2, \quad r_V^* = +\infty, \quad r_V^0=0. $$ 
 The radius $r_\la$ of the ball supporting $\rho_\la$ is given by
 $$ r_\la = +\infty \quad  \text{if $\la \le 2$}, \quad 
 r_\la =\frac{1}{\sqrt{\frac{\la}{2}}-1} \quad \text{if $\la > 2$}.
 $$
Like in the Thomas-Fermi model (see \cite{Lieb-Simon77}), $\rho_\la$ is not compactly supported once $\la$ is below a threshold value. After tedious but straightforward computations,  we get
$$ \rho_\la =  \frac{\la \wedge 2}{4\pi} \frac1{r(1+r)^3} \, \L^3 \res B(0, r_\la) $$
and
$$ M_\infty(\la v) = \begin{cases} \la + 1 - 3 \, \sqrt{\frac{\la}{2}} + \frac1{ 3\, \sqrt{\frac{\la}{2}}} & \text{if $\la \ge2$} \\
	\frac{ \la^2} {12}  &\text{if $\la \le 2$.} \end{cases} $$
\end{example}

\begin{example}\label{V4}
 Let $V(r) = \frac1{\sqrt{r}} \wedge 1$. Then
$$   \alpha_V=+\infty, \quad \la_V=0, \quad r_V^*= +\infty, \quad r_V^0=0.$$
In view of \autoref{easytest}, the potential $V$ is strongly confining, \ie, $\rho_\la\in \Prob(\R^3)$ for every $\la>0$.  The repartition function of $\rho_\la$
 is given by  
 $$  F_\la(r) = \begin{cases}
 	\frac{\la}{4} \sqrt{r} \wedge 1 & \text{ if $r\ge 1$} \\
 	0 & \text{if $r< 1$.}
 \end{cases} $$
Accordingly $r_\la = \max \left\{\frac{16}{\la^2}, 1\right\}$ and the probability  $\rho_\la$ reads
$$ \rho_\la = \begin{cases}  \frac{\la}{16\pi} \, \Big(  \frac{1}{2}\, r^{-\frac5{2}}\,  \L^3 \res \Sigma_\la + \mathcal{H}^2 \res \{|x|=1\}  \Big) & \text{if $\la < 4$} \\
  \frac{1}{4\pi}\, \mathcal{H}^2 \res \{|x|=1\} & \text{if $\la \ge  4$} ,
 \end{cases} $$
where, for $\la<4$, we denote  $\Sigma_\la := \{1< |x|< \frac{16}{\la^2}\}$.
 After some computations, we get
 $$ M_\infty(\la v) = \begin{cases} 
 	\frac{\la^2}{16} \log \frac{16}{\la^2} + \frac{3 \la^2}{16}   & \text{if $\la \le 4$} \\
 	\la - 1   & \text{if $\la \ge  4$.}
 \end{cases} $$

Notice that for $\la<4$, we have a mixed situation where $\rho_\la$  concentrates a fraction of its mass  on the interface $|x|=2$ while the other part is distributed in the crown $\Sigma_\la$. This kind of phenomenon was already observed for an anistropic cost function $\ell$ in case of a quadratic confining potential (see for instance \cite{carillo}).
\end{example}

\section{Continuum models as limits of small range interaction costs}\label{crowd}

In this section we study  the asymptotic  behavior of the multi-marginal cost $C_N$ when the two-particles repulsive cost $\ell$ is rescaled with respect to a small  parameter $\eps$ which tends to zero while $N\to \infty$. More precisely, we propose to replace $\ell$ with $\ell_\eps(t) \eqdef \ell(t/\eps)$ in the definition \eqref{def:c_N} of $c_N$, thus  leading to the small range interaction cost:
\begin{equation}\label{def:cNeps}
  c_N^\eps(x_1, \dotsc, x_N) = \frac{2}{N(N-1)} \sum_{1 \leq i < j \leq N}  \ell \pa{\frac{\abs{x_i-x_j}}{\eps}}.
\end{equation}
Motivated by the passage from discrete to continuous models arising for instance in crowd or traffic congestion, we will assume that the particles stay in the closure $\Ob$ of a bounded Lipschitz domain $\O\subset\R^d$ and we
propose to scale $\eps$ so that
\begin{equation}\label{def:kappa}
	N\eps^d \to \kappa
\end{equation}
where  $\kappa$ is a positive  parameter which  roughly represents  the limit volume as $\eps\to 0$ of a cluster of $N$ particles placed at the nodes of a $\eps Q$-periodic grid, where $Q = (-\frac{1}{2},\frac{1}{2})^d$ is the unit cube of $\R^d$. 

Choosing $N$ to be the driving parameter as in the previous sections, we set 
\[ \eps = \eps_N \eqdef \pa{\frac{\kappa}{N}}^{1/d}. \]
Accordingly, the functional on $\Prob(\Ob)$ under study will be
\begin{equation} \label{def:FNeps}
	F_N(\rho) \eqdef \inf \gra{\int c_N^{\eps_N}(x_1, \dotsc, x_N) \ dP(x_1, \dotsc, x_N) \st P \in \Pi_N(\rho)}.
\end{equation}

Our aim is to characterize the $\Gamma$-limit of $F_N$ as a functional on $\Prob(\Ob)$. It is natural to expect such a $\Gamma$-limit to be the restriction to $\Prob(\Ob)$ of a local functional defined on absolutely continuous measures  that is of the form $F_\infty(\rho) = \int_\O f(\frac{d\rho}{dx})\, dx$ for a suitable convex l.s.c integrand $f$. This guess is confirmed in a particular case on which we focus now on.
The general case is, in our opinion, a very challenging issue that deserves further studies.

\medskip
Henceforth we assume that the repulsive potential is given by
\begin{equation} \label{hardpot}
	\ell(t) = \begin{cases} 0 & \text{if $t \geq 1$} \\ +\infty & \text{otherwise.}\end{cases}
\end{equation}

Note that this function $\ell$ does not satisfy the local integrability assumption (H4) given in the introduction. The corresponding interaction energy given \eqref{def:cNeps} is related to the so called {\em  hard sphere model}, in which congruent spheres of diameter $\eps$ centered at $x_i$ are packed in a container in such a way that they do not overlap each other.

Despite its simplicity this hard-sphere repulsive potential is used in physics for understanding the equilibrium and dynamical properties of a variety of materials, including simple fluids, colloids, glasses,  granular media \cite{torquato2006packing}. We suggest that it can be also used in the crowd motion modeling in order to justify the passage from the discrete to the continuous level, with a possible link with congestion  transport theory as  depicted  in \cite{maury2018congested}.

In this crowd model the particles or individuals have a given minimal distance $\eps$ to each other and are located on discrete subset $\Sigma$ of a container $\Ob$. To any such configuration we may associate the empirical measure 
\[ \mu_\Sigma \eqdef \frac{1}{\sharp(\Sigma)} \sum_{x\in \Sigma} \delta_x. \]

Accordingly, we may define another functional on $\Prob(\Ob)$ by setting
\begin{equation} \label{def:FNmu}
	\widetilde{F_N}(\rho) = \begin{cases} F_N(\rho)  & \text{if $\rho = \mu_\Sigma$ for some $\Sigma \subset \Ob$} \\ +\infty & \text{otherwise.} \end{cases}
\end{equation} 
As can be readily checked, $\widetilde{F_N}$ is non-convex and larger than $F_N$. Observe that both functionals $F_N$ and $\widetilde{F_N}$ are indicator functions, meaning that they take values in $\gra{0,+\infty}$, due to the particular choice of the cost function $\ell$ in \eqref{hardpot}. Moreover, they share the same Fenchel conjugate as we have:
\begin{equation} \label{dualFN}
	F_N^*(v) = \widetilde{F_N}^*(v) = \sup \gra{S_N v(x_1, \dotsc, x_N) \st x_i \in \Ob, \abs{x_i - x_j} \geq \eps_N \text{ if $i \neq j$}}.
\end{equation}
Therefore, since $F_N$ is convex and weakly lower semicontinuous, we have $\widetilde{F_N}^{**} = F_N$, which means that $F_N$ coincides with the convex l.s.c. envelope of $\widetilde{F_N}$. Note that the supremum in \eqref{dualFN} is actually a maximum, as $\Ob^N$ is compact.
 
\medskip   
Our convergence result involves a constant related to the densest sphere packing volume fraction in $\R^d$ namely:
\begin{equation} \label{def:packing}
	\gamma_d \eqdef \inf_{k\in \N^*} \frac{S(Q_k)}{k^d} = \lim_{k\to \infty} \frac{S(Q_k)}{k^d},
\end{equation}
 where $Q_k = \left[ -\frac{k}{2}, \frac{k}{2} \right]^d$ and, for any Borel set $A\subset\R^d$, $S(A)$ denotes the maximal number of points in $A$ with mutual distance larger or equal to $1$, that is (denoting $\Delta$ the diagonal of $\R^d$)
\begin{equation} \label{def:S}
	S(A) \eqdef \sup \gra{\sharp(\Sigma) \st \Sigma\subset A, \abs{x-y} \geq 1 \ \forall (x,y) \in \Sigma^2\setminus \Delta} 
\end{equation}
Further details related to this constant $\gamma_d$ will be given later. 
 
A specific feature of the confined  hard spheres model is that the parameter $\kappa$ introduced in \eqref{def:kappa} cannot exceed a threshold which depends of $\gamma_d$ and of the volume of the container. More precisely, the congestion ratio $\theta$ defined by 
\begin{equation} \label{def:theta}
	\theta \eqdef \frac{\kappa}{\gamma_d |\O|}
\end{equation}
is required to be not larger than 1. In virtue of \autoref{congestion} below, this condition is necessary to have that the $\Gamma$-limit of $(F_N)$ (resp $\widetilde{F_N}$) is not identically $+\infty$. 

\begin{thm}\label{micro-macro}
Let $\ell$ be given by \eqref{def:FNeps} and assume that $\theta$ given by \eqref{def:theta} satisfies $\theta\in [0,1)$. Then the sequences $(F_N)$ and $(\widetilde{F_N})$ defined  by \eqref{def:FNeps} and \eqref{def:FNmu} respectively $\Gamma$-converge to the indicator function
of the set
$$ \mathcal{K} \eqdef \gra{ \rho\in \Prob(\Ob) \ \st\   \rho \ll \L^d\res\O\ ,\  \frac{d\rho}{d\L^d}\,  \le\, \frac1{\theta \, |\O|} \text{ a.e.}}. $$

In other words, for every absolutely continuous $\rho\in \Prob(\Ob)$, the shared $\Gamma$-limit is given by $F_\infty(\rho) = \int_{\Ob} f(\frac{d\rho}{dx}) dx$ where $f$ is the indicator of the interval $[0, \frac1{\theta |\O|}]$, while $F_\infty(\rho) = +\infty$ otherwise. 
\end{thm}

As a corollary, we have the convergence of the Wasserstein distance $W_2(\rho, \mathcal{K}_N) = \min \gra{W_2(\rho, \nu) : \nu \in  \mathcal{K}_N}$, where
\[ \mathcal{K}_N \eqdef \gra{\mu_\Sigma \st \Sigma\subset \Ob, \sharp(\Sigma) = N, |x - y| \ge \e_N \text{ on $\Sigma_N^2\setminus\Delta$}}. \]

\begin{corollary}\label{distance}
For every $\rho \in \Prob(\Ob)$, it holds \ $W_2(\rho,\mathcal{K}_N) \to W_2(\rho,\mathcal{K})$ as $N\to\infty$.
\end{corollary} 

Before presenting the proof, we need some results linked with the packing constant $\gamma_d$. First we mention that the equality in \eqref{def:packing} is a consequence of a classical result in ergodic theory (see for instance \cite[Thm 2.1]{licht2002global}) which  applies to any set function $S$ on Borel subsets of $\R^d$  which is translation invariant and subadditive on disjoint sets, \ie, $S(A\cup B) \le S(A) + S(B)$ whenever $A\cap  B=\emptyset$. Observe moreover that, for our $S$ given in \eqref{def:S}, the subadditivity property above applies also for intersecting subsets.

Given a bounded Borel subset $A\subset \R^d$, we define for every $\eps > 0$: 
\begin{align*}
&n_\eps(A) \eqdef \sup \gra{ \sharp(\Sigma) \st \Sigma\subset A, \abs{x-y} \geq \eps \ \forall (x,y)\in\Sigma^2\setminus \Delta} \\
&\tilde n_\eps(A) \eqdef \sup \gra{\sharp(\Sigma) \st \Sigma\subset A, \abs{x-y} \geq \eps \ \forall (x,y)\in(\Sigma^2\setminus \Delta) \cup (\Sigma\times \partial A)}.
\end{align*}

Note that, since $A$ is bounded, the suprema above are finite, hence they are both maxima: $n_\e(A)$ is the maximal number of points in $A$ with mutual distance $\ge\e$, while ${\tilde n_\e}(A)$ denotes the maximal number of non overlaping open balls of diameter $\e$ contained in $A$. Obviously $n_\e$ and ${\tilde n_\e}$ are non decreasing set functions and ${\tilde n_\e} \le n_\e$.

It can be readily checked that $n_\e(Q_a) = S(Q_{a/\e})$ for every $a>0$ and that
\[ n_\e(A\cup B) \leq n_\e(A) + n_\e(B) \quad \text{for all Borel sets $A, B$}. \]
In contrast, the set function ${\tilde n_\e}$ is super-additive on disjoint sets, \ie,
\[ {\tilde n}_\e(A\cup B) \geq {\tilde n}_\e(A) + {\tilde n}_\e(B) \quad \text{ whenever $A\cap B = \emptyset$}. \]

\begin{lemma}\label{packinglemma}  Let $A$ be a bounded Borel subset of $\R^d$ with non empty interior such that $|\partial A| = 0$. Then
\begin{equation}
	\label{gamma-d} \lim_{\e\to 0} \e^d n_\e(A) = \lim_{\e\to 0} \e^d \tilde{n}_\e(A) = \gamma_d |A|. 
\end{equation}

Furthermore, for every $\e>0$ let $\Sigma_\e\subset A$ be an optimal subset for $ n_\e(A)$ (resp. for $\tilde{n}_\e(A)$). 
Then the associated empirical measure $\mu_\e= \frac1{n_\e} \sum_{i=1}^{n_\e} \delta_{x_i}$ converges tightly to the uniform probability density on $A$
as $\e\to 0$.
\end{lemma}


%

\begin{proof}
First we show \eqref{gamma-d} when $A = Q_a$ for $a>0$. Observe that $n_\e(Q_a) = S(Q_{\frac{a}{\e}})$ while $S(Q_{k_\e}) \leq S(Q_{\frac{a}{\e}}) \leq S(Q_{k_\e+1})$ being $k_\e$ the integer part of $a/\e$. Then the equality $\lim_{\e\to 0} \e^d n_\e(Q_a) = \gamma_d a^d$ follows since, by \eqref{def:packing}, we have  $S(Q_{k_\e})\sim \gamma_d\, k_\e^{-d}$ as $\e\to 0$.

On the other hand, for $\e < \delta < a$, we have $n_\e(Q_{a - \delta}) \leq {\tilde n_\e}(Q_a) \leq n_\e(Q_{a})$. Thus by applying the previous convergence to $Q_a$ and $Q_{a - \delta}$ and then sending $\delta\to 0$, we deduce that $\lim_{\e\to 0} \e^d\, {\tilde n_\e}(Q_a) = \gamma_d a^d.$     

Let now $A$ be a Borel set with non empty interior $\mathring{A}$ such that $|\partial A|=0$. We consider the family of hypercubes $\mathcal{A} \eqdef \gra{Q(x, a) \st x \in \ov A, a>0}$ being $Q(x,a)= x + Q_a$. As $\ov A$ is compact, for every $\delta > 0$ we may find a finite subfamily $\{Q_i: i\in I\} \subset \mathcal{A}$ such that $\ov A \subset \cup_{i\in I} Q_i $ and $\sum_{i\in I} |Q_i| \leq |\ov A| + \delta$. By the subadditivity property of the set function $n_\e$ and by using the first step, we get
\[ \limsup_{\e\to 0} \e^d n_\e( A) \leq \sum_{i\in I} \limsup_{\e \to 0} \e^d n_\e(Q_i) \leq \gamma_d \sum_{i\in I} |Q_i| 
\leq \gamma_d (|\ov A| + \delta) = \gamma_d (|A| + \delta), \]
where we used also that $|\partial A| = 0$; by sending $\delta\to 0$ we get
\[ \limsup_{\e\to 0} \e^d n_\e( A) \leq \gamma_d |A|. \]

On the other hand by Besicovitch's Covering Theorem (generalized by Morse to families of hypercubes, see for instance \cite{morse1947perfect}), there exists a countable family $\gra{Q_n} \subset \mathcal{A}$ such that $Q_n \subset \mathring{A}$, $Q_m \cap Q_n =\emptyset$ if $m \neq n$, with $\left|\mathring{A} \setminus \cup_n Q_n \right| = 0.$  By the superadditivity of the set function ${\tilde n_\e}$ and by applying first step to each $Q_n$, we deduce that
\[ \liminf_{\e\to 0} \e^d {\tilde n}_\e( A) \geq \sum_n \liminf_{\e\to 0} \e^d n_\e( Q_n) \geq \gamma_d \sum_{i\in I} |Q_i| \geq \gamma_d |\mathring{A}|\ . \]
Recalling that $\tilde{n}_\e( A) \leq n_\e(A)$ and that $\partial A$ was assumed to be of vanishing Lebesque measure, we deduce
the converse inequality  $ \liminf_{\e\to 0} \e^d n_\e(A) \ge \gamma_d |A|,$
whence  \eqref{gamma-d}.

\medskip
Let us now prove the second assertion of \autoref{packinglemma}.  Let $\Sigma_\e$ be an optimal subset of $A$ associated with $n_\e(A)$. Up to a subsequence we may assume that $\mu_\e$ converges tightly to a probability measure $\mu$ supported on the compact set $\ov A$. For every $x_0 \in \R^d$  and $\delta > 0$, by applying \eqref{gamma-d}, we find that: 
\[ \mu (Q_\delta(x_0)) \leq \liminf_{\e\to 0} \frac{\sharp(\Sigma_\e \cap Q_\delta(x_0))}{n_\e(A)} \leq \limsup_{\e\to 0}  \frac{n_\e(Q_\delta(x_0)\cap A)}{n_\e(A)} = \frac{| Q_\delta(x_0)\cap A|}{|A|}. \]

Since $x_0$ and $\delta$ are arbitrary, we infer that $\mu$ is an absolutely continuous measure with a density  $\frac{d\mu}{dx} \le \frac{1}{|A|}$ a.e. in $\ov A$. Thus as $\mu(\ov A)=1$ (and $|\ov A|=|A|$), we conclude that $\mu$ is the  uniform probability density on $A$ and it is the unique cluster point of $\mu_\e$ as $\e\to 0$.
In view of the first equality in \eqref{gamma-d}, the previous arguments work as well when substituting $n_\e$ with $\tilde{n}_\e$.
\end{proof}

\begin{remark} \label{packingequivalence}
In view of \autoref{packinglemma}, the relation between $\gamma_d$ given in \eqref{def:packing} and the best density of spheres packing constant in $\R^d$ is now clear since the maximal volume fraction of non overlapping spheres of diameter $\e$ which can be placed in a regular subset $A$ of $\R^d$, namely $\tilde{n}_\e(A) \,\omega_d \pa{\frac{\e}{2}}^d$ (being $\omega_d$ the volume of the unit sphere), is asymptotically equal to $\gamma_d\, \frac{\omega_d}{2^d}$.
The exact value of this volume fraction is well known for $d \leq 3$, where the optimal configuration can be recovered from a periodic lattice (regular hexagonal lattice for $d=2$). We refer to \cite{conway1999sphere} for a survey in more dimensions of the celebrated sphere packing problem.										
\end{remark}


As a consequence of \autoref{packinglemma} we have 

\begin{lemma}\label{congestion} Let $\ell$ be the repulsive potential as defined in \eqref{hardpot} and assume that there exists $(\rho_N)\in \Prob(\Ob)$ such that $\limsup_{N\to\infty} F_N(\rho_N) < +\infty$. Then the parameter $\kappa$ defined in \eqref{def:kappa} satisfies the inequality $\kappa \leq \gamma_d \abs{\O}$. 
Conversely, if the previous  inequality is strict, then  $\limsup_{N\to \infty} \inf \widetilde{F_N} < +\infty$. 
\end{lemma}

\begin{proof}
Since $F_N$ is an indicator function, the finiteness of $F_N(\rho_N)$ implies that $\inf F_N= -F_N^*(0) =0$  which in view of \eqref{dualFN} is equivalent to the existence of a subset $\Sigma \subset \Ob$ such that $\sharp(\Sigma)=N$ and $|x-y| \ge \e_N$ for all $(x,y) \in \Sigma^2\setminus\Delta.$ Therefore it is necessary that $N\le n_{\e_N} (\Ob)$ for $N$ large. By applying \eqref{gamma-d} to $A=\Ob$ and to the sequence $\e_N$, we get $n_{\e_N}(\Ob) \sim \e_N^{-d} \gamma_d |\O|$ as $N\to \infty$, thus concluding to the desired inequality since, with the help of \eqref{def:kappa}, we have
\[ \lim_{N\to \infty} \frac{N}{n_{\e_N}(\Ob)} = \frac{\kappa}{\gamma_d |\O|} \leq 1. \]

Conversely, if $\frac {\kappa}{\gamma_d |\O|} < 1$, then the inequality  $N \le n_{\e_N}(\Ob)$ holds for large $N$, ensuring the existence of a set $\Sigma\subset \Ob$ such that $\sharp(\Sigma)=N$ and $\widetilde{F_N}(\mu_\Sigma) = 0$.
\end{proof}

\begin{proof}[Proof of \autoref{micro-macro}] Since $F_N \le \widetilde{F_N}$, it is enough to establish the $\Gamma$-$\liminf$ inequality for $(F_N)$ and the $\Gamma$-$\limsup$ inequality for $(\widetilde{F_N})$. We proceed in two steps:

\med
{\bf Step 1} ({\em $\Gamma$-$\liminf$ inequality}).\ As $F_N$ is an indicator function, proving the $\Gamma$-$\liminf$ inequality amounts to show that, for any sequence $(\rho_N)$ such that $F_N(\rho_N) = 0$ and $\rho_N \wconv \rho$ in $\Prob(\Ob)$, it holds $\rho \in \mathcal{K}$.
Let $\f \in C^0_+(\Ob)$. Then, by applying the Fenchel inequality, we have: 
\begin{equation}\label{Fenchel}
 \bra{v, \rho} = \lim_{N\to \infty} \bra{v, \rho_N} \leq \limsup_{N\to \infty} F_N^*(v).
\end{equation}
By selecting an optimal subset $\Sigma_N \subset \Ob$ in \eqref{dualFN}, we deduce that
\[ F_N^*(v) = \bra{v, \mu_{\Sigma_N}} \ =\ \int_0^{+\infty} \frac{\sharp \left( \gra{v > t} \cap \Sigma_N \right)}{N} dt, \]
where in the last  equality we used the fact that $\mu_{\Sigma_N}(\gra{v > t}) = \frac{\sharp \left(\gra{v > t} \cap \Sigma_N \right)}{N}$.

Next we invoke the definition of the set function $n_{\e_N}$ to infer that:
\[ F_N^*(v) \leq \int_0^{+\infty} \frac{n_{\e_N} \left(\gra{v > t} \right)}{N} dt. \] 
By observing that, for a.e. $t \in [0, \sup v]$, the open subset $\gra{v > t}$ is non empty with a Lebesgue negligible boundary, we apply \eqref{gamma-d} to pass to the limit $N\to\infty$ in the previous equality with the help of dominated convergence:
\begin{ieee*}{rCl}
\limsup_{N \to \infty} F_N^*(v) & \leq & \lim_{N\to\infty} \int_0^{+\infty} \frac{ n_{\e_N} \left(\gra{v > t} \right)}{N} dt \\
& = & \int_0^{\sup v}  \frac{\gamma_d}{\kappa} |\gra{v > t}| dt = \frac{\gamma_d}{\kappa} \int_\O v(x) dx.
\end{ieee*}

Therefore, from \eqref{Fenchel}, we see that the inequality $\bra{v, \rho} \leq \frac{\gamma_d}{\kappa} \int_\O v(x) dx $ holds for every
non negative continuous test function $v$. It follows that $\rho$ is an absolutely continuous probability on $\Ob$ with a density $u$ such that $u \leq \frac{\gamma_d}{\kappa}$ a.e. This implies that $\rho\in \mathcal{K}$ since $\frac{\gamma_d}{\kappa} = \frac{1}{\theta |\O|}.$

\med
{\bf Step 2. $\Gamma$-$\limsup$ inequality}  We need to show that $\mathcal{K} \subset \mathcal{K}_\infty$ where
\[ \mathcal{K}_\infty \eqdef \gra{\rho \in \Prob(\Ob) \st \exists \Sigma_N\subset \Ob, F_N(\mu_{\Sigma_N}) = 0, \mu_{\Sigma_N} \wconv \rho}. \]
Since $\mathcal{K}_\infty$ is a closed subset of $\Prob(\Ob)$ equipped with the tight convergence, it is enough to 
show that $\mathcal{K}_0 \subset \mathcal{K}_\infty$ being $\mathcal{K}_0$ a dense subset of $\mathcal{K}$, namely (see \autoref{K0} below):
\begin{equation}\label{stepwise}
\mathcal{K}_0 = \gra{\rho\in \Prob(\Ob)\ \st\  \rho = \left(\sum_{i\in I} t_i \One_{A_i}\right) dx\ ,\ t_i\le u^*}, 
\end{equation}
where $u^*:= \frac1{\theta |\O|}$ and $\{A_i, i\in I\}$ is a finite family of disjoint open subsets such that 
$A_i\subset\subset \O$ and $|\partial A_i| = 0$.

Let $\rho \in \mathcal{K}_0$ be in the form above and denote $\omega = \cup_{i\in I} A_i$, where $I = \gra{1, \dotsc, K}.$
By setting $A_0 = \Omega\setminus\ov{\omega}$, we obtain a partition $\cup_{i=0}^K A_i$ of full measure in $\Omega$. 
Note that the condition
$\sum_{i\in I} t_i |A_i| =1$ with $t_i\le u^*$ implies that $1 \le u^* |\omega|$. Thus the volume ratio of $\omega$ 
satisfies $\frac{|\omega|}{|\O|} \geq \theta$ (recall that $\theta < 1$).

%
Let us  fix a parameter $\eta < 1$ that  ultimately will be sent to $1$. Then, for every $i\in I$, we set: 
\begin{equation}\label{delta-i}
\d_i \eqdef \left(\frac{ u^*}{\eta t_i}\right)^{1/d}, \quad  \e_{i,N} = \d_i \e_N \quad \text{(thus $\e_{i,N}>\e_n$)}\, .
\end{equation}

According to \eqref{gamma-d}, for any $i\in I$, there exists a set $\Sigma_{i,N} \subset A_i$ such that:
\[ \bigcup_{x\in \Sigma_{i,N}} B(x,\e_{i,N}) \subset A_i, \quad \sharp(\Sigma_{i,N}) = \tilde{n}_{\e_{i,N}}(A_i) \sim \frac{\gamma_d}{\e_N^d \d_i^d} |A_i|, \quad \mu_{\Sigma_{i,N}} \wconv \frac{\One_{A_i}}{|A_i|}. \]
Then we set $\Sigma'_N: = \cup_{i\in I} \Sigma_{i,N}$ and, taking into account \eqref{def:kappa}, \eqref{delta-i} and the equality $\sum_{i\in I} t_i |A_i| = 1$, we obtain:
\[ \lim_{N\to \infty} \frac{\sharp(\Sigma'_N)}{N} = \eta, \quad \frac{\sum_{x\in \Sigma'_N} \delta_x}{N} \wconv \eta \rho. \]

Since $\eta<1$, we know that $\sharp(\Sigma'_N) < N$ for large $N$ and we need completing $\Sigma'_N$ with a subset $\Sigma_{0,N}\subset A_0$ such that $\sharp(\Sigma_{0,N})= N - \sharp(\Sigma'_N)$ and $\{ B(x,\e_N) \st x\in \Sigma_{0,N}\}$ is a family of disjoints balls in $A_0$. This requires that  $N - \sharp(\Sigma'_N) \le \tilde{n}_{\e_N}(A_0)$. This condition  is indeed satisfied  for $\eta$ close to $1$ and large $N$ since, by \eqref{gamma-d}, it holds $ \lim_{N\to \infty} \frac{\sharp(\Sigma_{0,N})}{N}= 1-\eta$ while $ \lim_{N\to \infty} \frac{\tilde{n}_{\e_N}(A_0)}{N} = u^* |A_0|$. Possibly passing to a subsequence, we may assume that $\mu_{\Sigma_{0,N}} \wconv \rho_{0,\eta}$ where $\rho_{0,\eta} \in \Prob(\ov{A_0})$.

Summarizing we have obtained a set $\Sigma_N = \Sigma'_N \cup \Sigma_{0,N}$ such that $\sharp(\Sigma_N) = N$ and $\mu_{\Sigma_N} = (1 - \eta) \mu_{\Sigma'_N} + \eta \mu_{\Sigma_{0,N}}$. In addition $\mu_{\Sigma_N}$  is admissible since by construction
 $\{B(x,\e_n) \st x\in \Sigma_N\}$ is a family of disjoint balls contained in $\Omega$. Therefore the weak limit of $\mu_{\Sigma_N}$ as $N\to\infty$ namely $\rho_\eta \eqdef \eta \rho + (1-\eta) \rho_{0,\eta}$ belongs to $\mathcal{K}_\infty$. Since this is true for any $\eta<1$ close to $1$, by sending $\eta\to 1$ we conclude that $\rho\in \mathcal{K}_\infty$.
\end{proof}

\begin{lemma} \label{K0}
	Assume that $\theta$ given by \eqref{def:theta} satisfies $\theta <1$ and let $\mathcal{K}_0$ be the subset of $\mathcal{K}$ defined in \eqref{stepwise}. Then $\mathcal{K}_0$ is weakly* dense in $\mathcal{K}$.
\end{lemma}
  
\begin{proof} Let us consider first $\rho\in \mathcal{K}$ of the form $\rho = u dx$ where $u\in C(\O, [0,u^*]) $
 ($u^*= \frac1{\theta |\O|}$) and $\supp u\subset\subset\O$. Given $\d>0$, we select an increasing sequence of real numbers $t_0 = 0 < t_1 < \dots t_{k-1} < t_k = u^*$ such that $t_{i}-t_{i-1}\le \d$ and $|\{u=t_i\}|=0$ for all $i\ge 1$; then we set $\O_i= \gra{u > t_i}$ for $i\ge 0$ and $A_i = \O_{i-1} \setminus \ov{\O_i}$ for $1\le i\le k$.
	
We see that the $A_i$'s are non empty disjoint open subsets of $\O$ such that $\cup_{i=1}^k A_i$ is of full Lebesgue measure in $\O_0 = \gra{u > 0}$ and we have $\Omega_0 \Subset \O$ by our assumption on the support of $u$. Therefore we obtain an element $\rho_\delta = u_\delta dx \in \mathcal{K}_0$ by setting:
\[ u_\delta \eqdef \sum_{i=1}^k u_i \One_{A_i} \quad \text{where} \quad u_i = \frac{1}{|A_i|} \int_{A_i} u(x) dx. \]
As $|u_\d - u| \leq \d$, we clearly have $\rho_\d \wconv \rho$ hence $\rho \in \ov{\mathcal{K}_0}$.

By using the classical convolution approximation argument (which preserves the integral constraint $\int_\O u(x) dx = 1$ and the inequality $u \le u^*$), we can deduce that the same conclusion applies to all $\rho \in \mathcal{K}$ with compact support in $\O$.
The last step consists in showing that any element $\rho = u\ dx$ in $\mathcal{K}$ can be approximated by a sequence $\rho_n = u_n dx$ in $\mathcal{K}$
such that $\supp u_n \Subset \O$. Thanks to the assumption $u^* |\O| > 1$, the non negative function $u^*-u$ has a positive integral over $\O$.
Then we may take a non decreasing sequence of compact sets $K_n$ such that $\O= \cup_n K_n$ and consider the following sequence:
\[ u_n = \One_{K_n} (u + s_n (u^*-u))\quad \text{where}\quad s_n = \frac{\int_{\O\setminus K_n} u dx}{\int_{K_n} (u^* - u) dx}. \]

As $\lim_n \int_{K_n} (u^* - u_n) dx = \int_{\O} (u^*-u) dx > 0$, we have $s_n \to 0$ and  therefore $u_n \le u^*$ for large $n$. In addition $\int u_n dx = 1$ while $\supp u_n\Subset\O$ by construction. As $u_n\to u$ in $L^1(\O)$, we infer that $u \in \ov{\mathcal{K}_0}$.
\end{proof}

\bigskip

\appendix

\section{Tools and notation from convex analysis}\label{A}

Let $X$ be a topological vector space, and $f \colon X \to (-\infty, +\infty]$. The \emph{lower semi-continuous envelope} of $f$, denoted $\overline{f}$ is the greatest lower semi-continuous function below $f$, \ie,
\[ \overline{f}(x) = \sup \gra{g(x) \st g \leq f, g \ \text{ is l.s.c.}}  = \inf_{x_n \to x} \{\liminf_{n \to \infty} f(x_n)\}. \]

The \emph{convex hull} of $f$, denoted by $\cl f$, is the largest convex lower semi-continuous  function below $f$. It may be defined as the function whose epigraph is the closed convex hull of the epigraph of $f$ in $X \times \overline{\R}$. Notice that $f$ is lower semi-continuous iff $\ov{f}=f$ and that if $f$ is convex then  $\cl f = \ov{f}$. 

\noindent The \emph{Legendre-Fenchel conjugate} of $f$, denoted by $f^*$, is defined on $X^*$ as
\[ f^*(v) = \sup \gra{\bra{v,x} - f(x) \st x \in X}, \]
while its biconjugate , denoted  by $f^{**}$, is given by
\[ f^{**}(x) = \sup \gra{\bra{v,x} - f^*(v) \st v \in X^*}. \]
We list some well-known properties (see for instance \cite{bouchitte2006, zalinescu2002convex}) 
\begin{align} \label{prop-Fenchel} \begin{cases} 
(a)& \text{$f^*$ is convex and lower semi-continuous;} \\
(b)& \text{ if $f \leq g$, then $f^* \geq g^*$;  }  \\ 
(c)& \text{if  $\{f^*<+\infty\}$ is non empty, then $f^{**} = \cl f$ ;}\\
(d)& \text{$(\ov{f})^*= f^*$ and  $(\ov{f})^{**} = f^{**}$;}\\
(e)& \text{if $f$ is convex, then $\ov{f} =  f^{**}$.}\\
 \end{cases}
 \end{align}


Let $f \colon X \to (-\infty, +\infty]$ be a convex l.s.c. function and $x_0\in X$ such that  $f(x_0)<+\infty.$  
Then, for every $z\in X$, the map $t \mapsto \frac1{t} (f(x_0+tz)-f(x_0))$ is non decreasing on $(0,+\infty)$. We define  the {\em recession function} of $f$ as follows: 
$$ f_\infty(z) :=\sup_{t>0} \frac{f(x_0+tz)-f(x_0)}{t} = \lim_{t\to \infty} \frac{f(x_0+tz)-f(x_0)}{t}. $$ 
The next result shows that the definition above does not depend of $x_0$. 
\renewcommand{\thelemma}{\empty{}} 
  \begin{lemma}  $f_\infty$  coincides with the support function of the subset ${\rm dom} \, f^* := \{v\in X^* \st f^*(v)<+\infty\}$, that is:
\begin{equation}\label{recession}
 f_\infty(z) \ =\ \sup \{ \bra{v,x} \st v\in  {\rm dom} \, f^* \}.
\end{equation}
 Accordingly $f_\infty$ is convex l.s.c. and positively one-homogeneous on $X$.
\end{lemma}

\section{Some basic results in \texorpdfstring{$\Gamma$}{Gamma}-convergence theory}  \ For an extensive study  we refer
to  the monographs  \cite{attouch1984variational,braides2002gamma, dal2012introduction}. Recall that a sequence of functionals $(F_n)$ defined on a topological space $X$ is said to be $\Gamma-$ convergent to  $F: X\mapsto (-\infty, +\infty]$ if the two following conditions hold:
 \begin{enumerate}[i)]
\item \ ($\Gamma-\liminf$) \ If $x_n\to x$ in $X$, then $\liminf_{n\to\infty} F_n(x_n)\ge F(x)$.
\item ($\Gamma-\limsup$) \ For all $x\in X$, there is a recovering sequence $(x_n)$ in $X$ such that $x_n\to x$ and 
$\limsup_{n\to\infty} F_n(x_n)\le F(x)$.
\end{enumerate}

%
It is easy to show that if $F_n \Gconv F$, then $F$ is l.s.c on $X$ . Furhermore we have the following equivalence:
$$ F_N \Gconv F  \iff   \overline{F_n}  \Gconv F . $$
As a consequence, one can readily check that, if $(F_n)$ is monotone non decreasing, then we have 
$F_N \Gconv F$ where  $F = \sup_n  \ov{F_n}.$

\medskip
In addition we will use the two following results:
 \begin{prop}[{\cite[Proposition 5.9]{dal2012introduction}}] \label{prop:pointwise-gamma}
	Let $X$ be a separable Banach space, $(G_n)_{n\in\N}$ a sequence of equi-Lipschitz functionals on $X$. Then:
$$ G_n \Gconv G  \iff  G_n(x)\to G(x) \ ,\ \forall x\in X.$$
\end{prop}

\begin{prop} \label{prop:attouch-dual}
	Let $X$ be a separable Banach space  and $X^*$ its dual endowed with the weak* topology. Let $(F_n)_{n \in \N}$ be a sequence of convex   functionals  on $X^*$ satisfying the equicoercivity condition
	$$  \sup F_n(\rho_n) < +\infty \ \Rightarrow  \ \sup_n \norm{ \rho_n}_{X^*} <+\infty. $$
 Then the following  equivalence hold:
$$  F_N \Gconv F \quad \text{in $X^*$}\ \iff\  F_N^* \Gconv F^* \quad \text{in $X$} .$$
\end{prop}

\begin{proof}
	See \cite[Theorem 3.11]{attouch1984variational} in the reflexive case and D. Az\'e \cite{aze1986convergence}
 for a generalization to the non-reflexive case . \end{proof}

\bibliographystyle{plain}
\bibliography{biblio.bib}

\begin{thebibliography}{10}

\bibitem{attouch1984variational}
H.~Attouch.
\newblock {\em Variational convergence for functions and operators}.
\newblock Applicable Mathematics Series. Pitman (Advanced Publishing Program),
  Boston, MA, 1984.

\bibitem{aze1986convergence}
D.~Az\'{e}.
\newblock Convergence des variables duales dans des probl\`emes de transmission
  \`a travers des couches minces par des m\'{e}thodes d'\'{e}pi-convergence.
\newblock {\em Ricerche Mat.}, 35(1):125--159, 1986.

\bibitem{bouchitte2006}
G.~Bouchitt\'{e}.
\newblock {\em Convex Analysis and Duality}.
\newblock Encyclopedia of Mathematical physics. Academic Press, 2006.

\bibitem{bouchitte2020relaxed}
Guy Bouchitt{\'e}, Giuseppe Buttazzo, Thierry Champion, and Luigi De~Pascale.
\newblock Relaxed multi-marginal costs and quantization effects.
\newblock In {\em Annales de l'Institut Henri Poincar{\'e} C, Analyse non
  lin{\'e}aire}. Elsevier, 2020.

\bibitem{braides2002gamma}
Andrea Braides et~al.
\newblock {\em Gamma-convergence for Beginners}, volume~22.
\newblock Clarendon Press, 2002.

\bibitem{Brandts2020}
Jan Brandts, Sergey Korotov, and Michal K\v{r}\'{\i}\v{z}ek.
\newblock {\em Simplicial partitions with applications to the finite element
  method}.
\newblock Springer Monographs in Mathematics. Springer, Cham, [2020] \copyright
  2020.

\bibitem{buttazzo2018continuity}
Giuseppe Buttazzo, Thierry Champion, and Luigi De~Pascale.
\newblock Continuity and estimates for multimarginal optimal transportation
  problems with singular costs.
\newblock {\em Applied Mathematics \& Optimization}, 78(1):185--200, 2018.

\bibitem{carillo}
J.~A. Carrillo, J.~Mateu, M.~G. Mora, L.~Rondi, L.~Scardia, and J.~Verdera.
\newblock The equilibrium measure for an anisotropic nonlocal energy.
\newblock {\em Calc. Var. Partial Differential Equations}, 60(3):Paper No. 109,
  28, 2021.

\bibitem{choquet1958diametre}
Gustave Choquet.
\newblock Diam\`etre transfini et comparaison de diverses capacit\'es.
\newblock {\em S\'eminaire Brelot-Choquet-Deny. Th\'eorie du potentiel}, 3,
  1958-1959.
\newblock talk:4.

\bibitem{conway1999sphere}
J.~H. Conway and N.~J.~A. Sloane.
\newblock {\em Sphere packings, lattices and groups}, volume 290 of {\em
  Grundlehren der Mathematischen Wissenschaften [Fundamental Principles of
  Mathematical Sciences]}.
\newblock Springer-Verlag, New York, third edition, 1999.
\newblock With additional contributions by E. Bannai, R. E. Borcherds, J.
  Leech, S. P. Norton, A. M. Odlyzko, R. A. Parker, L. Queen and B. B. Venkov.

\bibitem{cotar2015infinite}
Codina Cotar, Gero Friesecke, and Brendan Pass.
\newblock Infinite-body optimal transport with coulomb cost.
\newblock {\em Calculus of Variations and Partial Differential Equations},
  54(1):717--742, 2015.

\bibitem{cotar2019next}
Codina Cotar and Mircea Petrache.
\newblock Next-order asymptotic expansion for n-marginal optimal transport with
  coulomb and riesz costs.
\newblock {\em Advances in Mathematics}, 344:137--233, 2019.

\bibitem{dal2012introduction}
Gianni Dal~Maso.
\newblock {\em An introduction to $\Gamma$-convergence}, volume~8.
\newblock Springer Science \& Business Media, 2012.

\bibitem{diaconis1980finite}
Persi Diaconis and David Freedman.
\newblock Finite exchangeable sequences.
\newblock {\em The Annals of Probability}, pages 745--764, 1980.

\bibitem{Frostman}
Otto Frostman.
\newblock Potentiel d'\'equilibre et capacit\'e des ensembles avec quelques
  applications \`a la th\'eorie des functions,lund.
\newblock {\em Doctoral thesis}, 115 s., 1935.

\bibitem{SL}
Thomas Lebl\'{e} and Sylvia Serfaty.
\newblock Large deviation principle for empirical fields of log and {R}iesz
  gases.
\newblock {\em Invent. Math.}, 210(3):645--757, 2017.

\bibitem{licht2002global}
Christian Licht and G\'erard Michaille.
\newblock Global-local subadditive ergodic theorems and application to
  homogenization in elasticity.
\newblock {\em Annales Math\'ematiques Blaise Pascal}, 9(1):21--62, 2002.

\bibitem{Lieb-Simon77}
Elliott~H. Lieb and Barry Simon.
\newblock The {T}homas-{F}ermi theory of atoms, molecules and solids.
\newblock {\em Advances in Math.}, 23(1):22--116, 1977.

\bibitem{dimarino2022grandcanonical}
Simone~Di Marino, Mathieu Lewin, and Luca Nenna.
\newblock Grand-canonical optimal transport.
\newblock {\em arXiv.2201.06859}, 2022.

\bibitem{maury2018congested}
Bertrand Maury.
\newblock Congested transport and microscopic and macroscopic scales.
\newblock In {\em European {C}ongress of {M}athematics}, pages 427--442. Eur.
  Math. Soc., Z\"{u}rich, 2018.

\bibitem{morse1947perfect}
Anthony~P. Morse.
\newblock Perfect blankets.
\newblock {\em Trans. Amer. Math. Soc.}, 61:418--442, 1947.

\bibitem{pass2013optimal}
Brendan Pass.
\newblock Optimal transportation with infinitely many marginals.
\newblock {\em Journal of Functional Analysis}, 264(4):947--963, 2013.

\bibitem{petrache2017next}
Mircea Petrache and Sylvia Serfaty.
\newblock Next order asymptotics and renormalized energy for {R}iesz
  interactions.
\newblock {\em J. Inst. Math. Jussieu}, 16(3):501--569, 2017.

\bibitem{Serfaty2015}
Sylvia Serfaty.
\newblock {\em Coulomb gases and {G}inzburg-{L}andau vortices}.
\newblock Zurich Lectures in Advanced Mathematics. European Mathematical
  Society (EMS), Z\"{u}rich, 2015.

\bibitem{serfaty2018systems}
Sylvia Serfaty.
\newblock Systems of points with {C}oulomb interactions.
\newblock In {\em Proceedings of the {I}nternational {C}ongress of
  {M}athematicians---{R}io de {J}aneiro 2018. {V}ol. {I}. {P}lenary lectures},
  pages 935--977. World Sci. Publ., Hackensack, NJ, 2018.

\bibitem{torquato2006packing}
Monica Skoge, Aleksandar Donev, Frank~H. Stillinger, and Salvatore Torquato.
\newblock Packing hyperspheres in high-dimensional {E}uclidean spaces.
\newblock {\em Phys. Rev. E (3)}, 74(4):041127, 11, 2006.

\bibitem{zalinescu2002convex}
Constantin Zalinescu.
\newblock {\em Convex analysis in general vector spaces}.
\newblock World scientific, 2002.

\end{thebibliography}

\end{document}